\documentclass[preprint,12pt]{elsarticle}

\usepackage{amssymb}
\usepackage{amsthm}

\usepackage[margin=1in]{geometry}
\usepackage[utf8]{inputenc}
\usepackage{macros}
\usepackage{graphicx}
\usepackage{booktabs}
\usepackage{xfrac}
\usepackage{tikz}
\usetikzlibrary{decorations.text}
\usetikzlibrary{arrows}
\usetikzlibrary{shapes}

\usetikzlibrary{matrix,fit}
\pgfkeys{tikz/mymatrix/.style={matrix of math nodes,inner sep=0pt,row sep=0em,column sep=0em,nodes={inner sep=6pt}}}
\pgfkeys{tikz/mat1/.style={
    matrix of math nodes,
    inner sep=0pt,
    row sep=0em,
    column sep=0em,
    nodes={inner sep=2.5pt}}
    }

\newcommand*\mymatrixbox[5][]{\node [fit= (m-#2-#3) (m-#4-#5)] [draw=blue,thick,solid,rounded corners,inner sep=+2pt,#1] {};}

\usepackage{parskip}
\allowdisplaybreaks

\journal{Journal of Combinatorial Theory, Series B}

\begin{document}

\begin{frontmatter}

\title{Eigenvectors of the De Bruijn Graph Laplacian:\\A Natural Basis for the Cut and Cycle Space}

\author[gv]{Anthony Philippakis\texorpdfstring{\corref{cor1}}{*}}
\ead{ap@gv.com}
\author[broad,cse]{Neil Mallinar\texorpdfstring{\corref{cor1}}{*}}
\ead{nmallina@ucsd.edu}
\author[iitb]{Parthe Pandit}
\ead{pandit@iitb.ac.in}
\author[hdsi]{Mikhail Belkin}
\ead{mbelkin@ucsd.edu}

\cortext[cor1]{Corresponding author(s)}

\affiliation[gv]{organization={GV},addressline={5 Cambridge Ctr}, 
            city={Cambridge},
            state={MA},
            postcode={02142}, 
            country={USA}}

\affiliation[broad]{organization={The Broad Institute of MIT and Harvard},addressline={415 Main St}, 
            city={Cambridge},
            state={MA},
            postcode={02142}, 
            country={USA}}

\affiliation[cse]{organization={Department of Computer Science and Engineering, UC San Diego},addressline={3235 Voigt Dr}, 
            city={San Diego},
            state={CA},
            postcode={92093}, 
            country={USA}}

\affiliation[iitb]{organization={Center for Machine Intelligence and Data Science, IIT Bombay},
addressline={KReSIT Building},
city={IIT, Powai, Mumbai},
state={Maharashtra},
postcode={400076},
country={India}}

\affiliation[hdsi]{organization={Halıcıoğlu Data Science Institute, UC San Diego},addressline={3234 Matthews Ln}, 
            city={San Diego}, 
            state={CA},
            postcode={92093},
            country={USA}}

\begin{abstract}
We study the Laplacian of the undirected De Bruijn graph over an alphabet $\cA$ of order $k$.  
While the eigenvalues of this Laplacian were found in 1998 by \citet{Delorme1998TheSO}, an explicit description of its eigenvectors has remained elusive.  
In this work, we find these eigenvectors in closed form and show that they yield a natural and canonical basis for the cut- and cycle-spaces of De Bruijn graphs.
Remarkably, we find that the cycle basis we construct is a basis for the cycle space of both the undirected \textit{and} the directed De Bruijn graph.
This is done by developing an analogue of the Fourier transform on the De Bruijn graph, which acts to diagonalize the Laplacian. 
Moreover, we show that the cycle-space of De Bruijn graphs, when considering all possible orders of $k$ simultaneously, contains a rich algebraic structure, that of a graded Hopf algebra.
\end{abstract}


\end{frontmatter}

\paragraph{Declaration of interest and funding sources:} Anthony Philippakis is employed as a Partner at GV (formerly known as Google Ventures) and has received funding from Microsoft, Alphabet, IBM, Intel, Bayer. 
Neil Mallinar is supported by the Eric and Wendy Schmidt Center at The Broad Institute. Parthe Pandit is supported by the Simons Institute for the Theory of Computing at UC Berkeley.
Mikhail Belkin is supported by the National Science Foundation (NSF) and the Simons Foundation for the
Collaboration on the Theoretical Foundations of Deep Learning\footnote{\url{https://deepfoundations.ai/}} through awards DMS-2031883 and \#814639 as well as NSF IIS-1815697 and the TILOS institute (NSF CCF-2112665).
No funding sources were involved in this research, the writing of the article, nor the decision to submit the article for publication.

\tableofcontents
\newpage

\section{Introduction}
\label{introduction}

In this work, we consider the class of directed graphs whose vertices are all ordered $k$-tuples over a given alphabet $\cA$, whose edges are all ordered $k+1$-tuples, and where a directed edge connects two $k$-tuples if the last $k-1$ letters of the source vertex match the first $k-1$ letters of the target vertex (see Figure 1).  Known as a de Bruijn (``dB") graph of order k, it is an object of fundamental importance to both pure and applied mathematics, appearing in fields as diverse as combinatorics, electrical engineering, and bioinformatics, to name just a few examples. They were first introduced by de Bruijn \citep{Bruijn1946ACP} and (independently) by Good \citep{Good1946NormalRD}, and they have been intensely studied since then.  
In particular, there has been a rich literature on the spectral theory of dB graphs, highlighted by papers in the 1990s that identified the eigenvalues of the dB graph Laplacian. \citep{StroSpectrum92,Delorme1998TheSO} 
However, while the eigenvalues of \dB graphs are now well understood, a closed-form characterization of their eigenvectors has remained elusive. 

This paper provides, to the best of our knowledge, the first closed-form description of the eigenvectors of the \dB graph Laplacian. 
Moreover, we use these eigenvectors as the starting point for a deeper analysis of the algebraic and spectral structure of dB graphs.  
In the prior works that analyzed the spectra of \dB graphs, the Laplacian and adjacency matrices are obtained by dropping the directional edge information from the directed \dB graph. 
In our work, we follow on this line of analysis.

The primary contributions of this work are the following:\\
\paragraph{(1) We simultaneously view words of length k and k+1 as the vertices and edges of a db graph, respectively, and as order k and k+1 tensor products of a vector space whose basis is indexed by the letters of our alphabet.}  This duality allows us to seamlessly translate between graph theory and multilinear algebra, leveraging results from each.  In particular, we demonstrate that essentially all of the spectral theory of dB graphs can be understood in terms of the actions of a pair of simple operators and their transposes on these tensor products, and we show that the incidence matrix, adjacency matrix, and Laplacian can all be expressed in terms of this pair of operators.\\
\paragraph{(2) We present a closed-form derivation of the eigenvectors of \dB graphs, and we show that they provide natural orthogonal bases for the cut space and cycle space of \dB graphs that have not previously been appreciated in the literature.}  Our approach relies on the aforementioned duality between dB graphs and tensor products;
while certain operations are not natural on graphs, they are very natural on tensor spaces.  One such operation is the discrete Fourier transform, which can be performed on a vector space after identifying the letters of our alphabet with elements of a finite group. This discrete Fourier transform naturally extends to an action on the tensor space.  We show that, in the same way that the Fourier transform diagonalizes the Laplace operator on $\R$, our discrete Fourier transform makes the dB graph Laplacian, $L_V$ a ``tri-diagonalized" operator.  From there, it's eigenvectors are easily recognized in Fourier space.

Moreover, it is well known that the set of all functions on the edge space of a graph can be decomposed into the \textit{cut space} (representing all cut sets on a graph) and it's orthogonal complement, the \textit{cycle space} (representing all cycles on a graph). 
Cuts and cycles of a graph are of fundamental importance in the fields of mathematics and computer science, and have been studied at great length in both theory and application.\footnote{See: \citet{Boykov2006,Yi2012ImageSA,snuil_cycleconnectivity,VASILIAUSKAITE2022127097,Diestel2005TheCS,Lewis2020AHA} for just a few such examples.}
Therefore, it is of great interest to understand the structure of these function spaces.
In our analysis, we show that the eigenvectors of the dB Laplacian can be naturally transformed into a natural basis for the cycle space and cut space of the graph.  These bases are orthogonal (unlike most bases for the cut and cycle space), and have a number of symmetries.
Furthermore, while our analysis uses the undirected \dB graph Laplacian, we find that the eigenvectors we derive are in the null space of the \textit{directed} incidence operator.
As such, they simultaneously form a basis for the cycle space of directed and undirected De Bruijn graphs.\\

\begin{figure}[!t]
    \centering
    \includegraphics[width=0.8\textwidth]{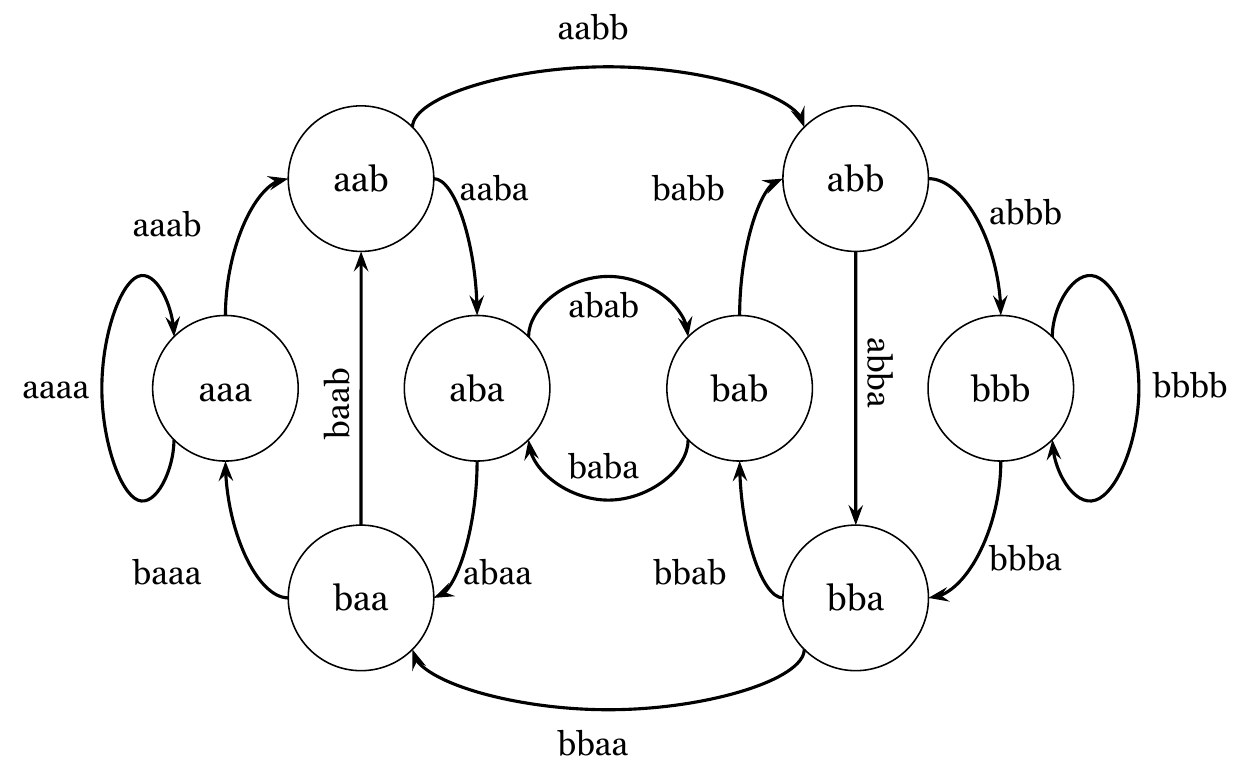}
    \caption{The $4^{th}$-order de Bruijn graph, $G_4$, with letters over the alphabet $\cA = \{a, b\}$.}
    \label{fig:G4_graph_example}
\end{figure}

\paragraph{(3) We demonstrate that, when considering the collection of dB graphs across all possible orders k, the collection of cycle spaces has a rich algebraic structure, that of a graded Hopf algebra.} This observation allows us to connect \dB graphs to the larger fields of representation theory and algebraic topology.
In particular, we find that the eigenvectors of the cycle space form a shuffle algebra.
These, and similar structures show up in various areas ranging from combinatorics \citep{Crossley2006SOMEHA,Patrias2015AntipodeFF,Grinberg2014HopfAI,Hazewinkel2000GeneralizedOS} to quantum field theory \citep{yeats_combinatorial}, and formalizing such connections should lead to a more generalized understanding of \dB graphs.\\ 

While the focus of our efforts is to better understand the spectral theory of dB graphs, it also opens the door to future applications in bioinformatics and natural language processing. In bioinformatics, the $k$-tuples are all subsequences of $k$ characters from the DNA alphabet, often referred to as $k$-mers.
The edges then represent $(k+1)$-mers and a weighted \dB graph encodes the edges with the occurrence count of $(k+1)$-mers from fragments of DNA reads.
A vector in the edge space of this \dB graph is referred to as a $k$-mer count vector.  
A number of important algorithms, ranging from assembling genomes to understanding gene regulation, are based on this $k$-mer count vector.  \citep{dBassembly,readmappingdB,applications_dB}
Similarly, in natural language processing these vectors are called ``bag-of-words" (BoW) vectors, where the ``alphabet" is over all words from some language, and a $k$-mer is the equivalent of what researchers  call an $n$-gram.
BoW vectors are fundamental objects in representing natural language and have driven much of the field of statistical and machine learning, in particular with respect to language data. 
\citep{Cavnar1994NgrambasedTC,Geiger2010TextCI,ALTINEL20181129,Lodhi2002TextCU}

In this manuscript, we show that $k$-mer count vectors are a lattice in the cycle space of \dB graphs.
Therefore, we can think of the eigenvectors of the \dB graph Laplacian as a basis for decomposing this vector.  Given that these eigenvectors play a role in functions on words that is analogous to that of sines and cosines on the real line, this opens the door to performing ``harmonic analysis'' on word frequencies. We therefore believe that this work provides the starting point for a fruitful framework for theorists and practitioners alike.
 \clearpage{}\subsection{Notation}
We provide a table summarizing symbols and notation used throughout the paper.
\begin{table}[ht!]
\centering
\resizebox{1.0\textwidth}{!}{\begin{tabular}{@{}ll@{}}
\toprule
Symbol & Description \\ \midrule
    $\cA$   &   finite alphabet with $|\cA|$ letters \\
    $q$ & $|\cA|$ $q-$ary alphabet \\
    $k$ & $k$-mer size \\
    $G$   &  directed or undirected graph           \\
    $V, V(G)$   &  vertices of a graph, $G$           \\ 
    $E, E(G)$  &   edges of a graph, $G$ \\
    $\cV(G)$    &   vertex space of a graph $G$, spanned by indicator vectors of vertices, dimension $|V(G)|$  \\
    $\cE(G)$    &   edge space of a graph $G$, spanned by indicator vectors of directed edges, dimension $|E(G)|$  \\
    $\cC(G)$    &   cycle space of a $G$, spanned by indicator vectors of directed cycles, dimension $|E(G)| - |V(G)| + 1$ \\
    $\cC(G)^\perp$    &   cut space of a $G$, dimension $|V(G)| - 1$ \\
    \midrule
    $\bfA$  &   adjacency matrix, $|V| \times |V|$ \\
    $\bfE$  &   incidence matrix, $|V| \times |E|$ \\
    $\bfL_V$    &   vertex Laplacian matrix $\bfE\bfE\tran$ \\
    $\bfL_E$    &   edge Laplacian matrix $\bfE\tran\bfE$\\\midrule
    $G_k$ & deBruijn graph of order $k$\\
    $|V(G_k)|$ & number of vertices in $G_k$, $q^{k-1}$\\
    $|E(G_k)|$ & number of edges in $G_k$, $q^{k}$\\
$\cV_\cA^k$   &   vector space $\Complex^{q^k}$\\
    $\cV(G_k)$    &   vertex space of a deBruijn graph $G_k$, dimension $\cV_\cA^{k-1}$  \\
    $\cE(G_k)$    &   edge space of a deBruijn graph $G_k$, dimension $\cV_\cA^k$  \\
    $\cC(G_k)$    &   cycle space of $G_k$, dimension $\cV_\cA^{k} - \cV_\cA^{k-1} + 1$ \\
    $\cC(G_k)^{\perp}$    &  cut space of $G_k$, dimension $\cV_\cA^{k-1} - 1$ \\\midrule
    $X_\cA^i$ &  Fourier subwords of length $i \leq k$. Set of basis vectors of $\cC(G)$. $|X_\cA^i| = \begin{cases} q^{i-2}(q- 1)^2 & i \geq 2\\
    q - 1 & i=1\\
    1 & i=0
    \end{cases}$, \\
$j$ & total padding for subwords such that $i + j = k$ \\
    $\ell$ & summation index over all paddings, $\ell \in \{1, \cdots, j+1\}$ for $\Lambda_V$ and $\ell \in \{0, \cdots, j+1\}$ for $\Lambda_E$ \\
    $h$ & indexes eigenvector-eigenvalue pairs for a given $(i, j)$ pair, $h \in \{1, \cdots, j+1\}$ for $\Lambda_V $ and $h \in \{0, \cdots, j+1\}$ for $\Lambda_E$\\
    \midrule
    $\theta_L, \hatT_L$ & left delete operator \& Fourier equivalent\\
    $\theta_R, \hatT_R$ & right delete operator \& Fourier equivalent\\
    $A, \hat{A}$ & adjacency operator \& Fourier equivalent\\
    $\inc, \hat{\inc}$ & incidence operator \& Fourier equivalent\\
    $\Lambda_V, \hLv$ & vertex Laplacian operator \& Fourier equivalent\\
    $\Lambda_E, \hLe$ & edge Laplacian operator \& Fourier equivalent\\
    \bottomrule
\end{tabular}}
\label{tab:notation}
\end{table}

\clearpage{}
\section{Preliminaries}

We first introduce a set of basic concepts in graph theory, multilinear algebra, and Fourier analysis that will be used throughout the paper. 
A key object in our work is that of a $k$-mer, which we define here.

\begin{definition}[$k$-mers] 
   For an alphabet $\cA$ composed of $q$ letters, we use $\cA^{k}$ to denote the set of all strings of length $k$ comprised of letters from $\cA$. A $k$-mer is any $s \in \cA^k$.
\end{definition} 

\subsection{Graph theory}

\subsubsection{Vector spaces of graphs}

Consider a \textbf{directed graph}, $G = (V, E)$ where $V = \{v_1, v_2, \cdots, v_n \}$ is a set of vertices and $E = \{e_1, e_2, \cdots, e_m \}$ is a set of edges, each of which is an ordered tuple $e = (v_1, v_2)$ indicating the vertex $v_1$ is the \textbf{``outgoing"} vertex of $e$ and $v_2$ is the \textbf{``incoming"} vertex.
The \textbf{in-degree} of a vertex is the number of edges for which it is an incoming vertex.
The \textbf{out-degree} of a vertex is the number of edges for which it is an outgoing vertex.
The \textbf{degree} of a vertex is the sum of the in-degree and out-degree of that vertex.

A \textbf{path} in the graph represents an ordered tuple of edges, $(e_1, e_2, \cdots, e_i)$ such that for any subsequent edges in the path, $e_j, e_{j+1}$, the incoming vertex on $e_j$ is the same as the outgoing vertex on $e_{j+1}$.
A graph is considered \textbf{connected} if there is a path between any two vertices in the graph.
A \textbf{cycle} in the graph represents a path $(e_1, e_2, \cdots, e_i)$ such that incoming vertex on $e_1$ is the same as the outgoing vertex on $e_i$. An \textbf{Eulerian path} in the graph is a path that visits every edge in the graph exactly once (vertices can be revisited). An \textbf{Eulerian cycle} is an Eulerian path that satisfies the conditions of being a cycle.

We now state Euler's Theorem and refer the reader to \citet{bollobas_modern} for further exposition and proof.
\begin{theorem}[Euler's Theorem]
\label{thm:eulers_thm_graph_cycle}
A directed, connected graph contains an Eulerian cycle \textit{if and only if} the degree of the in-degree is the same as the out-degree for every vertex in the graph.
\end{theorem}

The \textbf{adjacency matrix} of a graph $\mathbf{A}$ is an $n \times n$ matrix such that $A_{ij} = 1$ if $(v_i, v_j) \in E$ or $(v_j, v_i) \in E$, otherwise $A_{ij} = 0$. 
The \textbf{incidence matrix} of a graph $\mathbf{E}$ is an $m \times n$ matrix such that: $\mathbf{E}_{ji} = -1$ if $v_i$ is the outgoing vertex for edge $e_j$; $\mathbf{E}_{ji} = 1$ if $v_i$ is the incoming vertex for edge $e_j$; $\bfE_{ji} = 0$ otherwise.
The \textbf{vertex Laplacian matrix} of a graph $\bfL_V$ is an $n \times n$ matrix constructed as $\bfE^T \bfE$.
The \textbf{edge Laplacian matrix} of a graph $\bfL_E$ is an $m \times m$ matrix constructed as $\bfE \bfE^T$.

The \textbf{vertex \textit{vector} space} of a graph, $\cV(G)$, is a vector space spanned by basis elements representing each vertex of the graph.
The \textbf{edge \textit{vector} space} of a graph, $\cE(G)$, is a vector space spanned by basis elements representing each edge of the graph.

A \textbf{spanning subgraph}, $G'$ of a graph $G$, is a subgraph constructed from all of the vertices but only a subset of the edges of $G$, \ie,
$$G':=(V(G),E')\qquad E'\subseteq E(G)$$
The \textbf{cycle space} of $G$, denoted $\mc C(G)$, is the set of all spanning subgraphs of $G$ with each vertex having an even degree, or in the case of a directed graph the each vertex must have in-degree $=$ out-degree.
From \Cref{thm:eulers_thm_graph_cycle}, the cycle space of $G$ is the set of spanning subgraphs of $G$ that each contains an Eulerian cycle.

\begin{definition}
The dimension of the vector space representing the cycle space of $G$ $= m - n + c$ where $c$ is the number of connected components in $G$. This quantity is also known as the \textbf{circuit rank} of $G$, or the \textbf{cyclomatic number}, \textbf{cycle number}, or \textbf{nullity} of $G$.
\end{definition}

\subsubsection{De Bruijn graphs}

Let $s$ be a string of length $k$ given by $s = s_0 s_1 s_2 \cdots s_{k-1}$, where each $s_i \in \cA$.
Different $k$-mers will be referred to by superscripts, for example, $s^{(i)} = s^{(i)}_0s^{(i)}_1\cdots s^{(i)}_{k-1}$ and $s^{(j)} = s^{(j)}_0s^{(j)}_1\cdots s^{(j)}_{k-1}$ where $i, j \in \{0, 1, \cdots, |\cA^k|-1\}$.

\begin{definition}[De Bruijn graph]
The De Bruijn graph of order $k$ is denoted $G_k=(\cA^{k-1}, \cA^k)$ with vertices, $V(G_k)$, given by all strings $s \in \cA^{k-1}$ and edges between two strings $(s^{(i)}, s^{(j)})\in E(G_k)$ iff the last $k-2$ letters in $s^{(i)}$ are equal to the first letters in $s^{(j)}$. 
\end{definition}

Note here for the sake of notational simplicity we refer to the edge set as $\cA^k$, though it is not, rather it is isomorphic to $\cA^k$ in the following way: each edge in $G_k$ of the form $(s^{(i)}, s^{(j)})$ bijectively maps to the string $s^{(i)}_0s^{(i)}_1\cdots s^{(i)}_{k-1}s^{(j)}_{k-1} = s^{(i)}_0 s^{(j)}_0s^{(j)}_1\cdots s^{(j)}_{k-1}$ of length $k$.

$G_k$ has vertices representing $(k-1)$-mers and edges representing $k$-mers.
On an alphabet $\cA$, there are $|\cA|^{k-1}$ many $(k-1)$-mers and $|\cA|^k$ $k$-mers.
Therefore there are $n = |\cA|^{k-1}$ vertices in $G_k$ and $m = |\cA|^{k}$ edges in $G_k$.
By construction, we also have that $G_k$ has $c=1$ connected components and is a strongly connected, regular graph of degree $2$, with in-degree equal to the out-degree for every node.

Figure \ref{fig:G4_graph_example} depicts the \dB graph $G_4$ over the two-letter alphabet $\cA = \{A, B \}$. 

\subsection{Multilinear algebra}

In this paper we work primarily with circular strings, defined as follows.

\begin{figure}
  \centering
    \begin{tikzpicture}
        \draw[  postaction={         decorate,            decoration={                text along path,                text={abbbababbbababbabababaaabaabaaabbbbbabaaabbabaababbbbbaaababa},                text align={align=center},                raise=-2ex            }        }    ] (0,0) circle (2.13cm);
    \end{tikzpicture}
  \caption{Circular string composed of the letters $\cA = \{a, b\}$.}
  \label{fig:circ_string}
\end{figure}

\begin{definition}[Circular string] 
   A string of length $n$ given by $s = s_0s_1s_2\cdots s_{n-1}$ is called circular if it represents the underlying, infinite string $s_0s_1\cdots s_{n-1}s_0s_1s_2 \cdots$ obtained by repeating the string with periodicity $n$.
   Figure \ref{fig:circ_string} shows an example circular string over an alphabet of two letters.
\end{definition}

\subsubsection{Tensor product on vector spaces}

Let $\cV_\cA$ be the $|\cA|$-dimensional complex vector space, $\C^{|\cA|}$, with standard basis elements indexed by the letters of $\cA$.
For a given $k$, let 
$$\cV_\cA^k = \underbrace{\cV_\cA \otimes \cdots \otimes \cV_\cA}_{\textrm{$k$ times}}$$
where $\otimes$ is the tensor product operation.
Throughout this manuscript, when we refer to variables we use $s = s_0s_1\cdots s_{k-1}$ to refer to $k$-mer strings, where $s_i \in \cA$.
We will also use the variable $v = v_0 v_1\cdots v_{k-1}$ to refer to $k$-mer tensors in $\cV_\cA^k$.
Subscripts will refer to different letters or indices in the tensor product of the $k$-mer and it's representation in $\cA^k$ and $\cV_\cA^k$, respectively.
Superscripts will refer to different $k$-mers and tensors of $\cA^k$ and $\cV_\cA^k$ as needed.
For specific examples, we will use the two letter alphabet $\cA = \{a, b\}$.

\begin{example}[Tensor Product Basis] 
On $\cA = \{a, b\}$, the standard basis for $\VA^2$ would be given by the tensor products $$aa = a \otimes a \qquad ab = a \otimes b \qquad ba = b \otimes a \qquad bb = b \otimes b$$ and any element of $\VA^2$ can be expressed as a sum of these basis vectors.
\end{example}

For the interested reader, \citet{roman2007advanced} provides a thorough background on tensors and tensor products over vector spaces in Chapter 14.\footnote{\citet{nlab:tensor_product_of_vector_spaces} also provides a thorough, interactive website that covers many of the topics in this paper with respect to tensors and algebras.}

\begin{example}[Tensor product of vector spaces representing words]
    Let $\cA = \{a, b\}$, and $k=3$. This produces the following 8 $k$-mers: $aaa$, $aab$, $aba$, $abb$, $baa$, $bab$, $bba$, $bbb$. Assume $\VA$ has standard basis elements $a, b$, and $\VA^k$ is as defined above with standard basis elements given by $3$-tensor products on combinations of $a, b$. 

    Consider the (circular) word, $s$, given in Figure \ref{fig:circ_string}.
    The $k$-mer counts of this word is as follows:
    \begin{equation*}
    s: \quad aaa: 4 \quad aab: 7 \quad aba: 12 \quad abb: 6 \quad baa: 7 \quad bab: 11 \quad bba: 6 \quad bbb: 8
    \end{equation*}

    In $\R^8$ we can represent these counts with the vector $[4, 7, 12, 6, 7, 11, 6, 8]$. 
    In the tensor product basis on $\VA^3$, we can similarly represent this count vector with the tensor:
    \begin{align*}
        v: \qquad &4 (aaa) + 7 (aab) + 12 (aba) + 6 (abb) \\ + &7 (baa) + 11 (bab) + 6 (bba) + 8 (bbb)
    \end{align*}
\end{example}

\subsubsection{Operators over tensors}

We proceed to define a set of useful operators over tensors, but first require the following definition that is used extensively in our analysis.

\begin{definition}[de Bruijn vector, $\dBvec$]
    \label{def:dbvec}
    We refer to $\frac{1}{\sqrt{q}} \mathbf{1} \in \R^{q}$ as the ``de Bruijn vector" and denote it by $\dBvec$. In the tensor basis for $\cV_\cA^k$ the de Bruijn tensor is given by $\frac{1}{\sqrt{q^k}} \underbrace{(\dBvec \otimes \dBvec \cdots \otimes \dBvec)}_{k \text{ times}} \cong \frac{1}{\sqrt{q^k}} \mathbf{1} \in \R^{q^k}$.
    Note that if we assign each letter in $\cA$ to a standard basis vector $e_0, e_1, \cdots, e_{q-1}$, then $\dBvec$ is the virtual character with a unique representer in the vector space corresponding to $\cA$, given by $\dBvec = \frac{1}{\sqrt{q}}\sum_{a \in \cA} e_a = \frac{1}{\sqrt{q}}\sum_{a \in \cA} a$.
\end{definition}

\begin{example}[de Bruijn vector, $\dBvec$]
    Let $\cA = \{a, b\}$ with standard basis elements given as $a = \begin{pmatrix}
        1 \\ 0
    \end{pmatrix}$ and $b = \begin{pmatrix}
        0 \\ 1
    \end{pmatrix}$.
    Then $\dBvec = \frac{1}{\sqrt{2}}(a + b) = \frac{1}{\sqrt{2}}\begin{pmatrix}
        1 \\ 1
    \end{pmatrix}$.

    On $\VA^2$, we would then have
    \begin{align*}
        \dBvec\dBvec = \frac{1}{2}(a + b) \otimes (a+b) = \frac{1}{2}(aa + ab + ba + bb).
    \end{align*}
\end{example}

We now define two crucial operator overs $k$-mers.
Let $v \in \VA^k$ be a $k$-mer given by $v = v_0 v_1 v_2 \cdots v_{k-1}$.

\begin{definition}[Deletion operators] 
\label{def:std_delete}
For $k=1$, returns scalar 1, $k=0$, returns 0.
For $k\geq 2$, we have the left and right delete operators $\theta_L$ and $\theta_R$ as linear maps from $\VA^k\mapsto \VA^{k-1}$, defined as follows,
\begin{align}
    \theta_L(v) &= \frac{1}{\sqrt{|\cA|}} v_1\ldots v_{k-1}\in \VA^{k-1}\qquad\forall v \in\VA^k \\
    \theta_R(v) &= \frac{1}{\sqrt{|\cA|}} v_0\ldots v_{k-2}\in \VA^{k-1}\qquad\forall v \in\VA^k.
\end{align}
\end{definition}
These operators have their own adjoint operators, which are linear maps from $\VA^{k-1}\mapsto\VA^{k}$ given by, 
\begin{align}
    \theta^*_L(v) &= \dBvec v \in\VA^k \qquad\forall v \in\VA^{k-1} \\
    \theta^*_R(v) &= v \dBvec \in\VA^k \qquad\forall v \in\VA^{k-1}
\end{align}
where $\dBvec$ is the de Bruijn vector defined in Definition \ref{def:dbvec}.
Note that $\theta_R, \theta_L$ commute for all $k$.
We provide further intuitions and details on the deletion operators and their adjoints in \ref{apdx:deletion_ops}.
\newpage

We show here how the adjacency, incidence, and Laplacian matrices of the DeBruijn graph are represented using the deletion operators and their adjoints.
As in the case with $\theta_L, \theta_R$, the graph operators presented here simultaneously apply for all $k$, as opposed to the matrix form which is defined for a specific $k$.

We first present the matrix form of the adjacency, incidence, and Laplacian matrices.
Consider the DeBruijn graph $G_k = (\cA^{k-1}, \cA^k)$.
The adjacency matrix $\mathbf{A} \in \R^{\cA^{k-1} \times \cA^{k-1}}$ is constructed as follows: for vertices $v_1, v_2 \in \cA^{k-1}, \bfA_{v_1, v_2} = 1$ if $\exists$ a directed edge $(v_1, v_2)$ or $(v_2, v_1) \in \cA^k$, and $\bfA_{v_1, v_2} = 2$ if $v_1 = v_2$ (e.g. if there is a directed self-loop).
The incidence matrix $\bfE \in \R^{\cA^{k-1} \times \cA^{k}}$ is constructed as follows: for directed edge $(v_1, v_2) \in \cA^k, \bfE_{v_1, (v_1, v_2)} = -1$ and $\bfE_{v_2, (v_1, v_2)} = +1$. If $v_1 = v_2$ then $\bfE_{v_1, (v_1, v_1)} = 0$.
The vertex Laplacian $\mathbf{L}_V \in \R^{\cA^{k-1} \times \cA^{k-1}}$ is defined as $\mathbf{L}_V = \bfE \bfE^T = 2\mathbf{D} - \bfA = 2\mathbf{I} - \bfA$, as $G_k$ is a regular graph of degree 2.
The edge Laplacian $\mathbf{L}_E \in \R^{\cA^{k} \times \cA^{k}}$ is defined as $\mathbf{L}_E = \bfE^T \bfE$.

All of these matrices have equivalent definitions as operators, though we introduce new normalization factors in the operators for ease of bookkeeping later in the manuscript.
We define the corresponding operators here.
\begin{definition}[Adjacency operator]
    For $G_k = (\cA^{k-1}, \cA^k)$, the adjacency operator is defined as $\adj: \VA^{k-1} \to \VA^{k-1}$,
    \begin{equation}
        \adj = (\theta_R^* \theta_L + \theta_L^* \theta_R).
    \end{equation}
\end{definition}
We note that this is a different normalization than is standard for the adjacency matrix of a graph, that instead of having 1 in the adjacency you end up with $1/\cA$.

\begin{figure}
    \centering
    \includegraphics[width=0.8\textwidth]{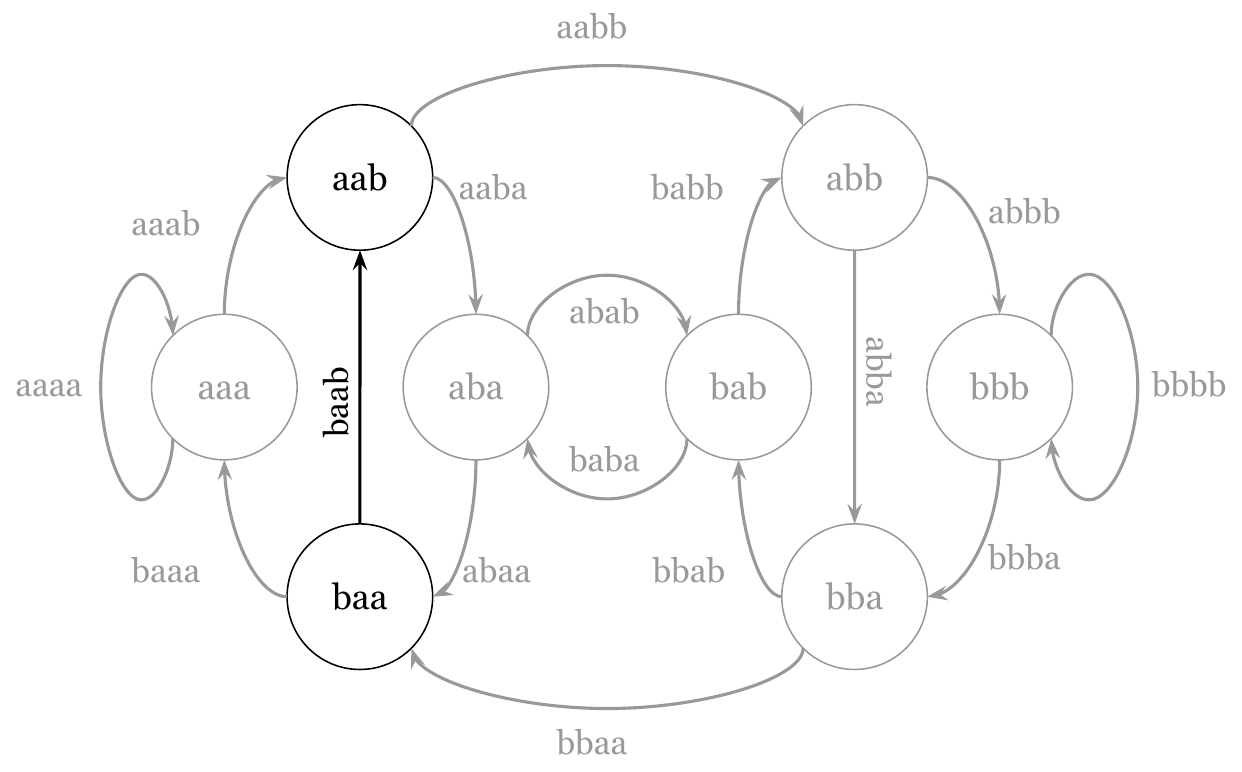}
    \caption{The De Bruijn graph $G_4$ on the alphabet $\cA = \{a, b\}$. We highlight two vertices and the directed edge between them. By definition, the incidence matrix encodes $+1$ on the element given by the edge and the incoming vertex, and $-1$ on the element given by the edge and the outgoing vertex. The incidence operator, similarly, acts as $\inc(baab) = \theta_L(baab) - \theta_R(baab) = aab - baa$.}
    \label{fig:incidence_example}
\end{figure}

\begin{definition}[Incidence operator]
    For $G_k = (\cA^{k-1}, \cA^k)$, the incidence operator is defined as $\inc: \VA^k \to \VA^{k-1}$,
    \begin{equation}
        \inc = \theta_L - \theta_R.
    \end{equation}

    It's adjoint is similarly defined as $\inc^*: \VA^{k-1} \to \VA^k$,
    \begin{equation}
        \inc^* = \theta_L^* - \theta_R^*.
    \end{equation}
\end{definition}
Observe that the operator $\inc$ is equivalent to a differently normalized version of the incidence matrix $\mathbf{E}$ where instead of $+1, -1$ in the incidence matrix, the operator outputs $+\frac{1}{\sqrt{\cA}}, -\frac{1}{\sqrt{\cA}}$.

\begin{definition}[Vertex Laplacian operator]
    For $G_k = (\cA^{k-1}, \cA^k)$, the vertex Laplacian operator is defined as, $\Lv: \VA^{k-1} \to \VA^{k-1}$,
    \begin{align*}
        \Lv &= \inc \inc^* \\
            &= (\theta_L - \theta_R)(\theta_L^* - \theta_R^*)\\
            &= (\theta_L\theta_L^* + \theta_R\theta_R^*) - (\theta_R^*\theta_L + \theta_L^*\theta_R) \\
            &= 2I - A.
    \end{align*}
\end{definition}

\begin{definition}[Edge Laplacian operator]
    For $G_k = (\cA^{k-1}, \cA^k)$, the edge Laplacian operator is defined as, $\Le: \VA^{k} \to \VA^{k}$,
    \begin{align*}
        \Le &= \inc^* \inc \\
            &= (\theta_L^* - \theta_R^*)(\theta_L - \theta_R)\\
            &= (\theta_L^*\theta_L + \theta_R^*\theta_R) - (\theta_R^*\theta_L + \theta_L^*\theta_R) \\
            &= (\theta_L^*\theta_L + \theta_R^*\theta_R) - A.
    \end{align*}
\end{definition}

Note: as mentioned above this is a slightly non-standard normalization for the graph operators, however it simplifies the bookkeeping with respect to constants of $|\cA|$ going forward.

\subsection{Fourier analysis}
\label{sec:dft}

\subsubsection{Discrete Fourier / Hadamard Transform}

Our construction of a natural orthogonal basis for the cycle- and cut-space of the de Bruijn graph will rely on a change of basis such that, rather than using the standard basis for the vector space, we instead use an orthogonal basis for $\VA^k$ where one of the vectors is the de Bruijn vector $\dBvec = \frac{1}{\sqrt{q}} \mathbf{1} \in \R^{q}$.  While the other basis vectors can be chosen arbitrarily (as long as they are orthonormal), it is especially natural to consider the following two choices for a change of basis, as they will act to diagonalize the de Bruijn graph Laplacian in the same way that the Fourier transform diagonalizes the laplacian on the real line:

\begin{enumerate}
    \item the Discrete Fourier Transform (DFT), $F_{q} \in \C^{q \times q}$, whose $i^{\text{th}}$ column is given by $F_{q}^{(i)} =$ 

    $\begin{pmatrix}
        \omega_q^{0\cdot i} &
        \omega_q^{1\cdot i} &
        \cdots &
        \omega_q^{(q-1) \cdot i}
    \end{pmatrix}^T$
    where $\omega_q = e^{-2\pi i / q}$ is a $q^{\text{th}}$ root of unity.
    \item the Walsh-Hadamard transform,
    $H_{q'} \in \R^{2^{q'} \times 2^{q'}}$, which in our case can be used if $\exists \ q' \geq 1$ such that $q = 2^{q'}$.
    The Hadamard transform for $q'=1$ coincides with $F_2$, in that $H_1 = F_2 = \mattwo{1}{1}{1}{-1}$. For $q' > 1$, we have the recursive definition $H_q = \begin{pmatrix}
        H_{q'-1} & H_{q'-1} \\
        H_{q'-1} & -H_{q'-1}
    \end{pmatrix}$
\end{enumerate}

For an element $v \in \VA^k$, the transformation to $\hat{v}$ is given by $\hat{v} = Fv = \underbrace{(F_q \otimes F_q \otimes \cdots \otimes F_q)}_{k \text{ times}}(v) = \hat{v}_0 \hat{v}_1 \cdots \hat{v}_{k-1}$.
We have analogous definitions for the transformation $v = F^{-1}\hat{v}$.\\

We can now define the right and left deletion operators as they act on an element in Fourier space; given $\hat{v} = \hat{v}_0\hat{v}_1\cdots\hat{v}_{k-1}$:
\begin{align}
    \hat{\theta}_L(\hat{v}) &= \begin{cases}
         \hat{v}_1\hat{v}_2\cdots\hat{v}_{k-1} & \text{if } \hat{v}_0 = \dBvec \\
        0 & \text{otherwise}
    \end{cases} \\
    \hat{\theta}_R(\hat{v}) &= \begin{cases}
         \hat{v}_0 \hat{v}_1 \cdots \hat{v}_{k-2} & \text{if } \hat{v}_{k-1} = \dBvec \\
        0 & \text{otherwise}
    \end{cases}
\end{align}

The adjoints of the Fourier deletion operators, $\hat{\theta}_L^*, \hat{\theta}_R^*$, are the same as $\theta_L^*, \theta_R^*$, respectively, except without the $\frac{1}{\sqrt{\cA}}$ normalization as this is absorbed by the Hadamard / Fourier transform.

\paragraph{Graph theory operators on Fourier basis}

All of the operators given in the prior section can be defined using the Fourier deletion operators as well. For example,
\begin{align*}
    \hat{A} &= (\hat{\theta}_R^* \hat{\theta}_L + \hat{\theta}_L^* \hat{\theta}_R) \\
            &= (F\theta_R^*F^{-1})(F\theta_L F^{-1}) + (F\theta_L^*F^{-1})(F\theta_R F^{-1}) \\
            &= F(\theta_R^*\theta_L + \theta_L^*\theta_R)F^{-1} \\
            &= FAF^{-1}.
\end{align*}
The operators $\hat{\inc}, \hLv, \hLe$ are defined analogously:
\begin{align*}
    \hat{\inc} = (\hatT_L - \hatT_R), \\
    \hat{\inc}^* = (\hat{\theta}_L^* - \hat{\theta}_R^*),\\
    \hLv = \hat{\inc}\hat{\inc}^* = 2I - \hat{A}, \\
    \hLe = \hat{\inc}^* \hat{\inc} = (\hat{\theta}_L^* \hatT_L + \hat{\theta}_R^* \hatT_R) - \hat{A}.
\end{align*}

\paragraph{Fourier words}
We define a key tensor space that we will use throughout the manuscript.
It is given by the intersection of the kernels of the incidence operators of \dB graphs and denoted by $X_\cA^{k} \subset V_\cA^k$.
Formally,
\begin{equation}
    X_\cA^k = \text{ker}(\hat{\theta}_L) \cap \text{ker}(\hat{\theta}_R).
\end{equation}
We call the members of $X_\cA^k$, \textit{Fourier words}.
Notationally, we will refer to elements of $X_\cA^k$ by the variable $\hat{x}$.
Recalling the definitions of $\hat{\theta}_L, \hat{\theta}_R$ we see that any $\hat{x} = \hat{x}_0 \hat{x}_1 \cdots \hat{x}_{k-1} \in X_\cA^k$ if $\hat{x}_0, \hat{x}_{k-1} \neq \dBvec$. 

 \section{Main results}
\label{sec:main_results}

\begin{figure}
    \centering
\begin{tikzpicture}[baseline=0cm, scale=0.25, transform shape]
    \matrix [mat1] (m)  
    {
    & aaa & aab & aba & baa & abb & bba & bab & bbb \\
    aaa & 2 & -1 & 0 & -1 & 0 & 0 & 0 & 0 \\
    aab & -1 & 4 & -1 & -1 & -1 & 0 & 0 & 0 \\
    aba & 0 & -1 & 4 & -1 & 0 & 0 & -2 & 0 \\
    baa & -1 & -1 & -1 & 4 & 0 & -1 & 0 & 0 \\
    abb & 0 & -1 & 0 & 0 & 4 & -1 & -1 & -1 \\
    bba & 0 & 0 & 0 & -1 & -1 & 4 & -1 & -1 \\
    bab & 0 & 0 & -2 & 0 & -1 & -1 & 4 & 0 \\
    bbb & 0 & 0 & 0 & 0 & -1 & -1 & 0 & 2 \\
    };
\end{tikzpicture}\qquad
\begin{tikzpicture}[baseline=0cm, scale=0.25, transform shape]
    \matrix [mat1] (m)  
    {
    & \dBvec\dBvec\dBvec & \dBvec\dBvec \hat{b} & \dBvec \hat{b} \dBvec & \hat{b} \dBvec \dBvec & \dBvec \hat{b}\hat{b} & \hat{b}\hat{b} \dBvec & \hat{b} \dBvec \hat{b} & \hat{b}\hat{b}\hat{b} \\
    \dBvec\dBvec\dBvec & 0 & 0 & 0 & 0 & 0 & 0 & 0 & 0 \\
    \dBvec\dBvec \hat{b} & 0 & 2 & -1 & 0 & 0 & 0 & 0 & 0 \\
    \dBvec \hat{b} \dBvec & 0 & -1 & 2 & -1 & 0 & 0 & 0 & 0 \\
    \hat{b} \dBvec \dBvec & 0 & 0 & -1 & 2 & 0 & 0 & 0 & 0 \\
    \dBvec \hat{b}\hat{b} & 0 & 0 & 0 & 0 & 2 & -1 & 0 & 0 \\
    \hat{b}\hat{b} \dBvec & 0 & 0 & 0 & 0 & -1 & 2 & 0 & 0 \\
    \hat{b} \dBvec \hat{b} & 0 & 0 & 0 & 0 & 0 & 0 & 2 & 0 \\
    \hat{b}\hat{b}\hat{b} & 0 & 0 & 0 & 0 & 0 & 0 & 0 & 2 \\
    };
\mymatrixbox{3}{3}{5}{5}
\mymatrixbox[red]{6}{6}{7}{7}
\mymatrixbox[darkspringgreen]{8}{8}{8}{8}
\mymatrixbox[darkspringgreen]{9}{9}{9}{9}
\end{tikzpicture}

    \caption{$\Lambda_V: \cV_\cA^3 \to \cV_\cA^3$ \textit{(left)} and $\hLv: \cV_\cA^3 \to \cV_\cA^3$ \textit{(right)} for $G_4$, e.g. $\cA = \{a, b\}$ with respective Fourier basis given by $\{\dBvec, \hat{b} \}, |\cA| = 2, k = 4$. 
    The Laplacian in native space, $\Lambda_V$, has no clear structure, whereas in Fourier space, $\hLv$, it decomposes into block matrices for each Fourier word.
    We highlight block $n \times n$ tri-diagonal Toeplitz matrices within $\hLv$, with $n=3$ in blue, $n=2$ in red, and $n=1$ in green.}
    \label{fig:example_vertex_laplacian}
\end{figure}

We now proceed to derive the eigenvectors of $\hLv, \hLe$.

\subsection{Eigenvectors of the vertex and edge Laplacians}

In this section we work with a de Bruijn graph of order $k$, $G_k$. First, we state a useful result and refer the reader to the proof given by \citet{tridiagonal}.

\begin{lemma}[Eigenvectors and eigenvalues of an $n \times n$ Topelitz matrix \citep{tridiagonal}]
\label{lem:tridiagonal}
The $n \times n$ tri-diagonal Toeplitz matrix takes the form:
$$
\begin{pmatrix}
    \delta & \tau & & & & \\
    \sigma & \delta & \tau & & \text{\huge0} & \\
    & \sigma & \delta & \tau & & \\
    & & \sigma & \delta & \tau & \\
    & \text{\huge0} & & \sigma & \delta & \tau \\
    & & & & \sigma & \delta 
\end{pmatrix} \in \R^{n \times n}
$$

The eigenvalues and eigenvectors of the $n \times n$ tri-diagonal Toeplitz matrix are well known.
The eigenvalues are given by
\begin{equation}
    \lambda_h = \delta + 2\sqrt{\sigma \tau} \cos \frac{h\pi}{n+1} \ \forall \ h \in \{1, \cdots, n\}
\end{equation}
and, for $\sigma\tau \neq 0$, the right-eigenvectors for each $h \in \{1, \cdots, n\}$ are given by
\begin{align}
    x_h &= [x_{h, 1}, x_{h, 2}, \cdots, x_{h, n}]^T \\
    x_{h, k} &= (\sigma/\tau)^{k/2} \sin \frac{hk\pi}{n+1} \ \forall \ k \in \{1, \cdots, n\}.
\end{align}
\end{lemma}

Before stating our main result, we provide a motivating example to show how the Fourier transform, taking $\Lv \to \hLv$, tri-diagonalizes the vertex Laplacian of $G_k$.
We will then generalize this to derive the eigenvectors of $\hLv, \hLe$.

\begin{example}
Consider, once again, the \dB graph for the case $\cA = \{a, b\}, k = 4$ given by $G_4$.
We denote the Fourier basis for $\cA$ by $\hat{a} = \dBvec$ and $\hat{b}$ by the orthogonal basis vector $\frac{1}{\sqrt{2}}\begin{pmatrix}
    1 & -1
\end{pmatrix}^T$.
The nodes of $G_4$ are given by $3$-mers with tensor basis elements $\dBvec\dBvec\dBvec, \dBvec\dBvec \hat{b}, \dBvec \hat{b} \dBvec, \hat{b} \dBvec\dBvec, \dBvec \hat{b}\hat{b}, \hat{b}\hat{b}\dBvec, \hat{b}\dBvec \hat{b}, \hat{b}\hat{b}\hat{b}$.
We compute the vertex Laplacian, $\hLv$, for this example in Figure \ref{fig:example_vertex_laplacian}.

Immediately, we see that the vertex Laplacian can be decomposed into tri-diagonal Toeplitz block matrices along the diagonal of $\hLv$. 
So the eigenvectors of $\hLv$ are given by the eigenvectors of each of these block matrices.

As given in Lemma \ref{lem:tridiagonal}, we obtain the following closed-form expressions for the eigenvectors of the $3 \times 3$ block in this example, in terms of the tensor basis elements for $G_4$:
\begin{align*}
    x_1 &= (\sin\dfrac{\pi}{4}) (\dBvec\dBvec \hat{b}) + (\sin\dfrac{2\pi}{4}) (\dBvec \hat{b} \dBvec) + (\sin\dfrac{3\pi}{4}) (\hat{b} \dBvec\dBvec) \\
    x_2 &= (\sin\dfrac{2\pi}{4}) (\dBvec\dBvec \hat{b}) + (\sin\dfrac{4\pi}{4}) (\dBvec \hat{b} \dBvec) + (\sin\dfrac{6\pi}{4}) (\hat{b} \dBvec\dBvec) \\
    x_3 &= (\sin\dfrac{3\pi}{4}) (\dBvec\dBvec \hat{b}) + (\sin\dfrac{6\pi}{4}) (\dBvec \hat{b} \dBvec) + (\sin\dfrac{9\pi}{4}) (\hat{b} \dBvec\dBvec)
\end{align*}

with eigenvalues:
\begin{align*}
    \lambda_1 &= 1 + 2\cos\frac{\pi}{4}, \quad \lambda_2 = 1 + 2\cos\frac{2\pi}{4}, \quad \lambda_3 = 1 + 2\cos\frac{3\pi}{4}.
\end{align*}

The eigenvectors and eigenvalues of the $2 \times 2$ block in this example are similarly given by:
\begin{align*}
    x_1 &= (\sin\dfrac{\pi}{3})(\dBvec \hat{b}\hat{b}) + (\sin\dfrac{2\pi}{3})(\hat{b}\hat{b}\dBvec),  &&\lambda_1 = 1 + 2\cos\dfrac{\pi}{3}\\
    x_2 &= (\sin\dfrac{2\pi}{3})(\dBvec\hat{b}\hat{b}) + (\sin\dfrac{4\pi}{3})(\hat{b}\hat{b}\dBvec),  &&\lambda_2 = 1 + 2\cos\dfrac{2\pi}{3}
\end{align*}
where we abuse notation by denoting these eigenvectors as $x_1, x_2$ again, namely to emphasize that each separate tri-diagonal block is modeled by the equations in Lemma \ref{lem:tridiagonal}.
However, these are not the same eigenvectors as obtained from the $3 \times 3$ block given above.

Finally, there are two eigenvectors for the $1 \times 1$ blocks given by,
\begin{align*}
    x_1 &= (\sin\dfrac{\pi}{2}) (\hat{b}\dBvec\hat{b}), &&\lambda_1 = 1 + 2\cos\dfrac{\pi}{2} \\
    x_1 &= (\sin\dfrac{\pi}{2}) (\hat{b}\hat{b}\hat{b}), &&\lambda_1 = 1 + 2\cos\dfrac{\pi}{2}.
\end{align*}

\end{example}
We see that the eigenvectors for each tridiagonal block are obtained from the elements of $\emptyset = X_\cA^0$, $\{\hat{b}\} = X_\cA^1$ and $\{\hat{b}\hat{b}\} = X_\cA^2$ being ``padded out" by the element $\dBvec$ three, two, and one time(s), respectively.
Note that this ``padding" operation corresponds to all the possible ways in which a $3$-mer can be obtained from elements of $X_\cA^i$ with $i < 3$.
This intuition will be important going forward, and we will formalize this padding operation in later sections.

We now state our main result, and provide some explicitly computed examples derived from it in \ref{app:sample_bases}.
\begin{theorem}[Eigenvectors of $\hLv, \hLe$]
\label{thm:eigs_lv_le}
Let $\hat{x} \in X_\cA^i$ for any $i \in \{1, \cdots, k\}$, and $j \in \Z_{\geq 0}$ such that $i + j = k$.

For each $\hat{x}$, we obtain a $(j+1) \times (j+1)$ block, tri-diagonal Toeplitz matrix in $\hLv: V_\cA^k \to V_\cA^k$.
We have a closed-form expression for the $j+1$ eigenvectors of each block, and thus the eigenvectors of $\hLv$ itself, indexed as $x_h, h \in \{ 1, \ldots, j+1 \}$.

These are given by,
\begin{equation}
    x_{h} = \sum_{\ell=1}^{j+1} \sin\bigg(\frac{\ell h\pi}{j+2}\bigg)(\hatT_L^*)^{j+1-\ell}(\hatT_R^*)^{\ell-1}\hat{x}
\end{equation}
and corresponding eigenvalue $\lambda_h = 2 - 2\cos\big( \frac{h\pi}{j+2} \big)$.
There is one eigenvector in the kernel of $\hLv$ given by $\dBvec$ with eigenvalue $0$ (this would correspond to the case of $\hat{x} \in X_\cA^0 = \emptyset$). 
Note that for any $\hat{x} \in X_\cA^k$ we get that $j=0$, therefore in such cases $h = 1$ and $x_1 = \hat{x}$ with eigenvalue $\lambda_1 = 2$.

For each $\hat{x}$, we obtain $j+2$ eigenvectors of $\hLe: V_\cA^{k+1} \to V_\cA^{k+1}$, indexed as $v_h, h \in \{ 0, 1, \ldots, j+1 \}$.

These are given by,
\begin{equation}
    v_{h} = \sum_{\ell=0}^{j+1} \cos\bigg( \frac{(2\ell+1)h\pi}{2(j+2)} \bigg) (\hatT_L^*)^{j+1-\ell}(\hatT_R^*)^{\ell} \hat{x}
\end{equation}
and the same, corresponding eigenvalue $\lambda_h = (2 - 2\cos\big(\frac{h\pi}{j+2}\big))$.
Note the following:
\begin{itemize}
    \item for $h=0, \lambda_0 = 0, v_0 \neq 0$, and so $v_0$ is in the kernel of $\hLe$
    \item in showing the above, we first prove that $v_0$ is in the kernel of $\hat{\inc}$
    \item we know $\dBvec \in \text{ker}(\hLe)$ and by definition, for each $\hat{x} \in X_\cA^{k+1}$, $\hat{x} \in \text{ker}(\hLe)$ and $\hat{x} \in \text{ker}(\hat{\inc})$.
\end{itemize}
\end{theorem}

\begin{proof}
    Let $x_h$ be as defined above with corresponding eigenvalue $\lambda_h = 2-2\cos(\frac{h\pi}{j+2})$ for some $\hat{x} \in X_\cA^i$. 
    We first compute $\inc^* x_h$, as this will be useful later. 
    For notational simplicity, we will let $\beta = \frac{h\pi}{j+2}$.

    \begin{align*}
        \hat{\inc}^* x_h &= (\hatT_L^* - \hatT_R^*) x_h 
        = (\hatT_L^* x_h - \hatT_R^* x_h) \nonumber \\
        &= \Bigg[\sum_{\ell=1}^{j+1} \sin(\beta \ell) (\hatT_L^*)^{j+2-\ell}(\hatT_R^*)^{\ell-1} \ \hat{x} - \sum_{\ell=1}^{j+1} \sin(\beta \ell) (\hatT_L^*)^{j+1-\ell}(\hatT_R^*)^{\ell} \ \hat{x} \Bigg] \nonumber \\
        &= \Bigg[ \sin\beta (\hatT_L^*)^{j+1} + \sum_{\ell=1}^{j} \Big[\sin(\beta (\ell+1)) - \sin(\beta \ell) \Big] (\hatT_L^*)^{j+1-\ell}(\hatT_R^*)^{\ell} - \sin(\beta (j+1)) (\hatT_R^*)^{j+1} \Bigg]\hat{x}. \nonumber
    \end{align*}
    Utilizing that $\sin(\phi_1) - \sin(\phi_2) = 2\cos(\frac{\phi_1 + \phi_2}{2})\sin(\frac{\phi_1 - \phi_2}{2})$, the inner terms become
    \begin{align*}
        \sin(\beta(\ell+1)) - \sin(\beta \ell) &= 2\cos(\frac{(2\ell + 1)\beta}{2})\sin(\frac{\beta}{2}).
    \end{align*} 
    Additionally, using $\sin(\phi) = 2\sin(\frac\phi2)\cos(\frac\phi2)$, we can simplify $\inc^* x_h =$
    \begin{align*}
        &\Bigg[ \sin(\beta) (\hatT_L^*)^{j+1} + 2\sin(\frac{\beta}{2}) \sum_{\ell=1}^{j} \cos\bigg( \frac{(2\ell+1)\beta}{2} \bigg) (\hatT_L^*)^{j+1-\ell}(\hatT_R^*)^{\ell} - \sin(\beta (j+1)) (\hatT_R^*)^{j+1} \Bigg]\hat{x} \\
        &= \sbrac{2\sin(\frac{\beta}{2})\sum_{\ell=0}^{j}\cos(\frac{(2\ell+1)\beta}{2})(\hatT_L^*)^{j+1-\ell}(\hatT_R^*)^{\ell} - \sin(\beta (j+1)) (\hatT_R^*)^{j+1}} \hat{x}\\
        &=(2\sin(\frac{\beta}{2})) \sbrac{\sum_{\ell=0}^{j+1}\cos(\frac{(2\ell+1)\beta}{2})(\hatT_L^*)^{j+1-\ell}(\hatT_R^*)^{\ell}} \hat{x}
    \end{align*}
    where the last step comes from,
    \begin{align*}
    -\sin(\beta (j+1) ) = -\sin(\beta (j+2-1) ) &= -\sin(h\pi-\beta) \\
    =(-1)^h\sin(\beta)=(-1)^h2\sin(\frac\beta2)\cos(\beta/2) &=2\sin(\frac\beta2)\cos(h\pi-\frac\beta2) \\
    = 2\sin(\frac\beta2)\cos((2(j+1)+1)\frac\beta2).
    \end{align*}
    Now we have that,
        \begin{align*}
            \hLv(x_h) &= && (\hatT_L - \hatT_R)\hat{\inc}^* x_h & &&\\
            &= && 2\sin\frac\beta2 \Big[ \cos\bigg(\beta\frac{2(j+1)+1}{2}\bigg)(\hatT_L)(\hatT_R^*)^{j+1} + \sum_{\ell=0}^{j} \cos\bigg(\beta\frac{2\ell+1}{2}\bigg)(\hatT_L^*)^{j-\ell}(\hatT_R^*)^{\ell} & &&\\
            & &&- \cos\frac\beta2 (\hatT_R)(\hatT_L^*)^{j+1} - \sum_{\ell=1}^{j+1} \cos\bigg(\beta\frac{2\ell+1}{2}\bigg) (\hatT_L^*)^{j+1-\ell}(\hatT_R^*)^{\ell-1} \Big] \hat{x} & &&\\
            &= && 2\sin\frac\beta2 \Big[ \sum_{\ell=1}^{j+1} \bigg( \cos\big(\beta\frac{2\ell-1}{2} \big) - \cos\big(\beta\frac{2\ell+1}{2}\big) \bigg) (\hatT_L^*)^{j+1-\ell}(\hatT_R^*)^{\ell-1} \Big]\hat{x} & &&
        \end{align*}
    where the last step comes from: (1) $\hatT_R, \hatT_L^*$ commute, as do $\hatT_L, \hatT_R^*$, and $\hatT_R\hat{x} = \hatT_L\hat{x} = 0$; (2) re-indexing $\sum_{\ell=0}^{j}$ as $\sum_{\ell=1}^{j+1}$. 
    Now, using that $\cos\phi_1 - \cos\phi_2 = -2\sin(\frac{\phi_1 + \phi_2}{2})\sin(\frac{\phi_1-\phi_2}{2})$:
    \begin{align*}
        \hLv(x_h) &= (2\sin\frac\beta2) \Big[ \sum_{\ell=1}^{j+1} -2\sin(\beta \ell)\sin\frac{-\beta}{2} (\hatT_L^*)^{j+1-\ell}(\hatT_R^*)^{\ell-1} \Big]\hat{x} \\
        &= (-4\sin\frac{\beta}{2}\sin\frac{-\beta}{2}) x_h \\
        &= \lambda_h x_h.
    \end{align*}
    Since $\frac{\beta}{2} = \frac{0 + \beta}{2}$ and $\frac{-\beta}{2} = \frac{0 - \beta}{2}$ we reuse the same identity to see that
    \begin{align*}
        \lambda_h &= 2(-2\sin\frac{\beta}{2}\sin\frac{-\beta}{2}) \\
        &= 2(\cos0 - \cos\beta) = 2(1 - \cos\beta)
    \end{align*}
    as desired. 
    Note here, the form of the eigenvalue as $\lambda_h = 2-2\cos\frac{h\pi}{j+2}$ is equivalent to re-indexing as $\lambda_{j+2-h} = 2 - 2\cos\frac{(j+2-h)\pi}{j+2} = 2-2\cos(\pi - \frac{h\pi}{j+2}) = 2 + 2\cos\frac{h\pi}{j+2}$, for $h \in \{1, \cdots, j+1\}$.

Now that we have established that $x_h, \lambda_h$ are eigenvectors and eigenvalues of $\hLv$, we notice that we can obtain a closed form expression for the eigenvectors of $\hLe$ as follows:
    \begin{align*}
        \hLv x_h &= \hat{\inc}\hat{\inc}^* x_h = \lambda_h x_h \\
        \hLe (\hat{\inc}^* x_h) &= \hat{\inc}^* \hat{\inc} (\hat{\inc}^* x_h) = \lambda_h (\hat{\inc}^* x_h).
    \end{align*}

    We have already calculated that
    \begin{align*}
        \hat{\inc}^*x_h &= 2\sin(\frac{\beta}{2}) \sum_{\ell=0}^{j+1} \cos(\frac{(2\ell + 1)\beta}{2}) (\hatT_L^*)^{j+1-\ell} (\hatT_R^*)^{\ell} \hat{x} \\
        &= 2\sin(\frac{\beta}{2}) v_h.
    \end{align*}

    Plugging in as described above, we see that
    \begin{align*}
        \hLe (\hat{\inc}^* x_h) &= \lambda_h (\hat{\inc}^* x_h) \\
        \Rightarrow 2\sin(\frac{\beta}{2}) \hLe v_h &= 2\sin(\frac{\beta}{2}) \lambda_h v_h \\
        \Rightarrow \hLe v_h &= \lambda_h v_h.
    \end{align*}
    Note that we established $\hLv x_h = \lambda_h x_h$ for $h \in \{1, 2, \cdots, j+1\}$ and so we also have that $\hLe v_h = \lambda_h v_h$ for the same values of $h$.
    For the case of $h=0$, we immediately observe that $\lambda_h = 0$.
    With regard to $\hLv$ we have that $x_0$ corresponds to the tensor with all zeros since $h=0$ will cause $\sin(\frac{\ell h\pi}{j+2})$ to be zero for any $\ell$, and so this trivially says that $\hLv (0) = 0$.
    However, we see that $v_0$ takes on the following form,
    \begin{align*}
        v_0 &= \sum_{\ell=0}^{j+1} \cos(\frac{(2\ell + 1)\cdot 0}{2(j+2)}) (\hatT_L^*)^{j+1-\ell} (\hatT_R^*)^{\ell} \hat{x} \\
        &= \sum_{\ell=0}^{j+1} (\hatT_L^*)^{j+1-\ell} (\hatT_R^*)^{\ell} \hat{x}
    \end{align*}
    which is not a zero-tensor. By simple calculation,
\begin{align*}
        \hLe v_0 &=  \hat{\inc}^* \hat{\inc} v_0 \\
        &=  \hat{\inc}^* \hat{\inc} \sum_{\ell=0}^{j+1} (\hatT_L^*)^{j+1-\ell} (\hatT_R^*)^{\ell} \hat{x} \\
        &= \hat{\inc}^* (\hatT_L - \hatT_R)\sum_{\ell=0}^{j+1} (\hatT_L^*)^{j+1-\ell} (\hatT_R^*)^{\ell} \hat{x} \\
        &= \hat{\inc}^* (\hatT_L - \hatT_R)[(\hatT_L^*)^{j+1}\hat{x} + (\hatT_R^*)^{j+1}\hat{x} + \sum_{\ell=1}^{j} (\hatT_L^*)^{j+1-\ell} (\hatT_R^*)^{\ell} \hat{x}] \\
        &= \hat{\inc}^*[ (\hatT_L^*)^j + \hatT_L (\hatT_R^*)^{j+1} + \sum_{\ell=1}^j (\hatT_L^*)^{j - \ell}(\hatT_R^*)^{\ell} \\ 
        &\quad -  \hatT_R (\hatT_L^*)^{j+1} - (\hatT_R^*)^{j} - \sum_{\ell=1}^j (\hatT_L^*)^{j+1-\ell} (\hatT_R^*)^{\ell-1} ] \hat{x} \\
        &= \hat{\inc}^*[ \hatT_L (\hatT_R^*)^{j+1} + \sum_{\ell = 1}^{j-1} (\hatT_L^*)^{j-\ell}(\hatT_R^*)^{\ell} - \hatT_R (\hatT_L^*)^{j+1} - \sum_{\ell=2}^j (\hatT_L^*)^{j+1-\ell}(\hatT_R^*)^{\ell-1} ] \hat{x} \\
        &= \hat{\inc}^*[ \hatT_L (\hatT_R^*)^{j+1} \hat{x} - \hatT_R (\hatT_L^*)^{j+1} \hat{x} ]
    \end{align*}
    where the last step comes from re-indexing summations to see that all of the inner terms cancel.
    Finally, we know that $\hatT_L$ commutes with $\thetaRt$, $\hatT_R$ commutes with $\thetaLt$, and $\hat{x} \in \text{ker}(\hatT_L) \cap \text{ker}(\hatT_R)$, resulting in:
    \begin{align*}
        \hLe v_0 &= \hat{\inc}^* ((\hatT_R^*)^{j+1} \hatT_L \hat{x} - (\hatT_L^*)^{j+1} \hatT_R \hat{x}) \\
        &=  \hat{\inc}^* (0) = 0.
    \end{align*}
    Therefore, $v_0 \neq 0$ and $\hLe v_0 = 0 = \lambda_0 v_0 \Rightarrow v_0 \in \text{ker}(\hLe)$.
    Note that in the steps above, we first showed that $\hat{\inc}v_0 = 0$, which means that $v_0$ is in the null space of the \textit{directed} incidence operator as well.
\end{proof}

\subsection{A basis for the cut space and cycle space}

Now that we have established the form of the eigenvectors for $\hLv$ and $\hLe$, we define an operator that will be useful going forward.

\begin{definition}[$\Xi^*_{j, h}$ operator]
We define the operator $\Xi^*_{j, h}: \cV_\cA^{k-j} \to \cV_\cA^k$ by,
\begin{align*}
    \Xi^*_{j, h}(v) = \frac{1}{\sqrt{j+1}} \sum_{\ell=0}^{j} \cos \Bigg( \frac{(2\ell+1)h\pi}{2(j+1)} \Bigg) (\thetaLt)^{j-\ell}(\thetaRt)^{\ell}(v)
\end{align*}
\end{definition}

The operator $\Xi^*_{j, h}$ works to ``pad out" words $v \in \cV_\cA^k$ with all possibilities of adding $j$ instances of $\dBvec$ to the left and right of $s$. 
The reason we define this operator in the adjoint form is to be notationally consistent with it's action via the adjoints $\theta_L^*, \theta_R^*$.
We provide an example to highlight this, and then show orthgonality of these operators.

\begin{example}[$\Xi^*_{j, 0}$ operator]
    Let $v \in V_\cA^k$ be any word and $|\cA| = q$. 
    Then $\Xi^*_{j,0}$ acts as follows:
    \begin{align*}
        \Xi^*_{1, 0}(v) &= \frac{1}{\sqrt{2}}(\dBvec v + v \dBvec) \\
        \Xi^*_{2, 0}(v) &= \frac{1}{\sqrt{3}}(\dBvec \dBvec v + \dBvec v \dBvec + v \dBvec \dBvec) \\
        \Xi^*_{3, 0}(v) &= \frac{1}{\sqrt{4}}(\dBvec \dBvec \dBvec v + \dBvec \dBvec v \dBvec + \dBvec v \dBvec \dBvec + v \dBvec \dBvec \dBvec) \\
        \cdots
    \end{align*}
\end{example}

\begin{lemma}[Orthogonality]
\label{lem:orthogonal}
For any $\hat{x} \in X_\cA^{k-j}, \hat{x}' \in X_\cA^{k-j'}$ with $0 \leq j, j' \leq k$, for all $h \in \{0, 1, \cdots, j+1\}, h' \in \{0, 1, \cdots, j'+1\}$, we have that $\Xi^*_{j,h}(\hat{x}) \perp \Xi_{j', h'}(\hat{x}')$ if: (1) $j \neq j'$; or (2) $j = j'$ and $\hat{x} \neq \hat{x}'$.

    \begin{proof}
        Let $\xi_j = \frac{1}{\sqrt{(j+1)}}$. Then we have,
        \begin{align*}
            &\langle \Xi^*_{j,h}(\hat{x}), \Xi^*_{j', h'}(\hat{x}') \rangle =\\
            &\quad \langle \xi_j \sum_{\ell=0}^j \cos \Bigg( \frac{(2\ell+1)h\pi}{2(j+2)} \Bigg)(\thetaLt)^{j-\ell}(\thetaRt)^\ell \hat{x}, \xi_{j'} \sum_{\ell'=0}^{j'} \cos \Bigg( \frac{(2\ell'+1)h'\pi}{2(j'+2)} \Bigg)(\thetaLt)^{j'-\ell'}(\thetaRt)^{\ell'} \hat{x}' \rangle \\
            &= \xi_j \xi_{j'} \sum_{\ell=0}^j \sum_{\ell'=0}^{j'} \langle \cos \Bigg( \frac{(2\ell+1)h\pi}{2(j+2)} \Bigg)  (\thetaLt)^{j-\ell}(\thetaRt)^\ell \hat{x}, \cos \Bigg( \frac{(2\ell'+1)h'\pi}{2(j'+2)} \Bigg) (\thetaLt)^{j'-\ell'}(\thetaRt)^{\ell'} \hat{x}' \rangle \\
            &= \xi_j \xi_{j'} \sum_{\ell=0}^j \cos \Bigg( \frac{(2\ell+1)h\pi}{2(j+2)} \Bigg) \sum_{\ell'=0}^{j'} \cos \Bigg( \frac{(2\ell'+1)h'\pi}{2(j'+2)} \Bigg) \langle  (\thetaLt)^{j-\ell}(\thetaRt)^\ell \hat{x}, (\thetaLt)^{j'-\ell'}(\thetaRt)^{\ell'} \hat{x}' \rangle.
        \end{align*}
    We then have the following cases:
    \begin{enumerate}
        \item $\ell < \ell'$:
        \begin{align*}
            \langle (\thetaLt)^{j-\ell}(\thetaRt)^\ell \hat{x}, (\thetaLt)^{j'-\ell'}(\thetaRt)^{\ell'} \hat{x}' \rangle &= \langle  (\thetaR)^{\ell'}(\thetaLt)^{j-\ell}(\thetaRt)^\ell \hat{x}, (\thetaLt)^{j'-\ell'} \hat{x}' \rangle \\
            &= \langle  (\thetaLt)^{j-\ell}(\thetaR)^{\ell'-\ell} \hat{x}, (\thetaLt)^{j'-\ell'} \hat{x}' \rangle \\
            &= \langle 0, (\thetaLt)^{j'-\ell'} \hat{x}' \rangle = 0
        \end{align*}

        \item $\ell > \ell'$:
        \begin{align*}
             \langle (\thetaLt)^{j-\ell}(\thetaRt)^\ell \hat{x}, (\thetaLt)^{j'-\ell'}(\thetaRt)^{\ell'} \hat{x}' \rangle &= \langle  (\thetaLt)^{j-\ell} \hat{x}, (\thetaR)^\ell(\thetaLt)^{j'-\ell'}(\thetaRt)^{\ell'} \hat{x}' \rangle \\
             &= \langle  (\thetaLt)^{j-\ell} \hat{x}, (\thetaLt)^{j'-\ell'}(\thetaR)^{\ell-\ell'} \hat{x}' \rangle \\
             &= \langle (\thetaLt)^{j-\ell} \hat{x}, 0 \rangle = 0
        \end{align*}

        \item $\ell = \ell'$:
        \begin{align*}
            \langle (\thetaLt)^{j-\ell}(\thetaRt)^\ell \hat{x}, (\thetaLt)^{j'-\ell'}(\thetaRt)^{\ell'} \hat{x}' \rangle &= \langle (\thetaLt)^{j-\ell} \hat{x}, (\thetaR)^\ell(\thetaLt)^{j'-\ell}(\thetaRt)^{\ell} \hat{x}' \rangle \\
            &= \langle (\thetaLt)^{j-\ell} \hat{x}, (\thetaLt)^{j'-\ell} \hat{x}' \rangle
        \end{align*}

        Similarly we can split this expression into cases based on $j, j'$,
        \begin{alignat*}{2}
            j > j': &\langle (\thetaLt)^{j-\ell} \hat{x}, (\thetaLt)^{j'-\ell} \hat{x}' \rangle &&= \langle \hat{x}, (\thetaL)^{j-j'} \hat{x}' \rangle = \langle \hat{x}, 0 \rangle = 0, \\
            j < j': & \langle (\thetaLt)^{j-\ell} \hat{x}, (\thetaLt)^{j'-\ell} \hat{x}' \rangle &&= \langle (\thetaL)^{j'-j} \hat{x}, \hat{x}' \rangle = \langle 0, \hat{x}' \rangle = 0, \\
            j = j': & \langle (\thetaLt)^{j-\ell} \hat{x}, (\thetaLt)^{j'-\ell} \hat{x}' \rangle &&= \langle (\thetaLt)^{j-\ell} \hat{x}, (\thetaLt)^{j-\ell} \hat{x}' \rangle \\
            & &&= \langle \hat{x}, \hat{x}' \rangle = 0 \text{ by definition for } \hat{x} \neq \hat{x}'.
        \end{alignat*}
    \end{enumerate}
    \end{proof}
\end{lemma}

\begin{corollary} We have that $\text{ker}(\hLe) = \text{ker}(\hat{\inc})$.

\begin{proof}
We show this result in two parts:

$\text{ker}(\hat{\inc}) \subseteq \text{ker}(\hLe)$: Let $v \in \text{ker}(\hat{\inc})$. Then $\hat{\inc}v = 0 \Rightarrow \hLe v = \inc^* \hat{\inc} v = \inc^*0 = 0$.

$\text{ker}(\hLe) \subseteq \text{ker}(\hat{\inc})$: 
Let $\hLv: \cV_\cA^k \to \cV_\cA^k$ and $\hLe: \cV_\cA^{k+1} \to \cV_\cA^{k+1}$.

In Theorem \ref{thm:eigs_lv_le} we proved that for each $\hat{x} \in X_\cA^i$ with $i \in \{1, \cdots, k\}$ and $j = k-i$, the element defined by $v_0 = \sum_{\ell=0}^{j+1} (\thetaLt)^{j+1-\ell}(\thetaRt)^{\ell}\hat{x} = (\sqrt{j+2}) \Xi^*_{j+1, 0}(\hat{x}) \in \text{ker}(\hLe)$.
To do so, we showed that $\hat{\inc} v_0 = 0$. 
Therefore, we immediately have that $v_0 \in \text{ker}(\hat{\inc})$.
The case of $X_\cA^0 = \emptyset$ corresponds to $v_0 = \dBvec$ and by the definition of $\dBvec, \hatT_L$, and $\hatT_R$, it is clear that $\dBvec \in \text{ker}(\hat{\inc})$.
And in the case of $\hat{x} \in X_\cA^{k+1}$, we similarly know by definition that $\hat{x} \in \text{ker}(\hLe)$ and $\in \text{ker}(\hat{\inc})$.

Therefore, we know that there is an eigenvector $\in \text{ker}(\hLe)$ for each $\hat{x} \in X_\cA^i$ for $i \in \{0, \cdots, k+1\}$, and that each of these eigenvctors are also $\in \text{ker}(\hat{\inc})$.
What remains to be shown is that there are no more elements of $\text{ker}(\hLe)$ other than the ones we have listed here.
In doing so, we then establish that for all $v \in \text{ker}(\hLe) \Rightarrow v \in \text{ker}(\hat{\inc})$.

In \ref{apdx:combinatorial_proof} we show, via a combinatorial proof, that $\sum_{i=0}^{k+1} |X_\cA^i| = q^{k+1} - q^{k} + 1$ where $q = |\cA|$.
Therefore, we have enumerated a total of $q^{k+1} - q^{k} + 1$ elements in $\text{ker}(\hLe)$.
By definition of $\hLe$ as a linear map, we know that $\text{dim}(\hLe) = |\cA^{k+1}| = q^{k+1}$.

To show that there can be no more elements in $\text{ker}(\hLe)$ we observe that, in order to derive the eigenvectors of $\hLe$ we started with eigenvectors of $\hLv$ that correspond to non-zero eigenvalues $\lambda_h$ for $h \in \{1, \cdots j+1\}$.
For each of these eigenvectors, we similarly obtained an eigenvector of $\hLe$ with non-zero eigenvalue, $\lambda_h$.
Since $\hLv$ is a linear map of dimension $q^{k}$ and there is only one vector, $\dBvec$, in the null space of $\hLv$ (see \citet{Veerman2020APO} for a recent exposition on the dimension of the null space of the vertex Laplacian) we get precisely $q^{k}-1$ eigenvectors of $\hLv$ (and thus of $\hLe$) with non-zero eigenvalue.

Putting it all together, we see that $\text{dim}(\text{ker}(\hLe)) + \text{dim}(\text{Im}(\hLe)) = (q^{k+1} - q^k + 1) + (q^k - 1) = q^{k+1} = \text{dim}(\hLe)$, as desired. 
Note that we have already established that these eigenvectors are orthogonal to each other in Lemma \ref{lem:orthogonal}, and so we have completely filled out a basis for $\hLe$.
As such, there are no elements of $\text{ker}(\hLe)$ that are unaccounted for, and we have shown that all such elements are $\in \text{ker}(\hat{\inc})$.

\end{proof}
\end{corollary}

We now show that the cycle space and cut space of $G_k$ can be decomposed into the direct sum of orthogonal subspaces, spanned by the eigenvectors derived in Theorem \ref{thm:eigs_lv_le}, which we now express in terms of the $\Xi_{j, h}$ operators.

\begin{definition}[Eigenvectors of $\hLe$ in terms of the operators $\Xi^*_{j, h}$]
Let $\hLe: \cV_\cA^{k} \to \cV_\cA^{k}$.

The cycle space, $W_\cA^k$, is spanned by the basis vectors of $\text{ker}(\hat{\inc}) = \text{ker}(\hLe)$.
The cut space, $(W_\cA^k)^{\perp}$, is spanned by the basis vectors of $\text{Im}(\hLe)$.
We proved in Theorem \ref{thm:eigs_lv_le} and Lemma \ref{lem:orthogonal} that elements given by $\Xi^*_{j, h}$ form an orthogonal basis for $\hLe$, where $\Xi^*_{j, 0}$ lie in the kernel and $\Xi^*_{j, h}, h \in \{1, \cdots, j\}$ lie in the image of $\hLe$, respectively.
Furthermore, we showed that $\Xi^*_{j, 0} \perp \Xi^*_{j, h}$ for $h \in \{1, \cdots, j\}$.

By Theorem \ref{thm:eigs_lv_le}, we have that
\begin{align*}
    W_\cA^k &= \Xi^*_{k, 0}(X_\cA^0) \oplus \Xi^*_{k-1, 0}(X_\cA^{1}) \oplus \Xi^*_{k-2, 0}(X_\cA^2) \oplus \cdots \oplus \Xi^*_{0, 0}(X_\cA^k) \\
(W_\cA^k)^{\perp} &= [\oplus_{h=1}^{k-1} \Xi^*_{k-1, h}(X_\cA^1)] \oplus [\oplus_{h=1}^{k-2}\Xi^*_{k-2, h}(X_\cA^2)] \oplus \cdots \oplus [\oplus_{h=1}^2 \Xi^*_{2, h}(X_\cA^{k-2})] \oplus [\Xi^*_{1, 1}(X_\cA^{k-1})].
\end{align*}

These are compactly written as $W_\cA^k = \oplus_{j=0}^k \Xi^*_{j, 0}(X_\cA^{k-j})$ and $(W_\cA^k)^{\perp} = \oplus_{j=1}^{k-1} [\oplus_{h=1}^j \Xi^*_{j, h}(X_\cA^{k-j})]$.

The edge space of $G_k$ is then given by $\cE(G_k) = W_\cA^k \oplus (W_\cA^k)^{\perp}$.
\end{definition}

In \ref{app:sample_bases} we give explicit examples of cut and cycle space bases. \section{The space of \texorpdfstring{$k$}{k}-mer counts}
\label{sec:space_of_kmer_counts}

We now present a result that motivates the importance of the cycle space, which we derive an algebraic structure for in the following section.

Let $\Omega$ be the set of all circular strings. 
\begin{definition}
For $\omega \in \Omega$, let $$\psi_k(\omega): \Omega \to \cV_\cA^k$$ be the operator that maps $\omega$ to it's $k$-mer count tensor. The image of this operator, for a given $k$, is the set of all possible count tensors in $\cV_\cA^k$.
\end{definition}

\begin{theorem}
Let $W_\cA^k = \textup{ker}(\theta_L - \theta_R) = \textup{ker}(\inc) \subset \cV_\cA^k$.
Then $\text{Im}(\psi_k(\Omega)) \subset W_\cA^k$.
\end{theorem}

\begin{proof}
Let $w \in \text{Im}(\psi_k(\Omega))$. Then $w$ has a representation as a weighted, linear combination of elements from the tensor basis on $\cV_\cA^k$. In other words, $w = \sum_{i=1}^{|\cA|^k} \alpha_i \kmeri$, recalling that $\kmeri$ is the tensor basis representation of a $k$-mer given by $\kmerstr{i}$ (assume lexicographic ordering of $k$-mers), and $\alpha_i$ represents the occurrence count of $\kmeri$.

Then we have that $\theta_R(w) = \sum_{i=1}^{|\cA|^k} \alpha_i \theta_R(\kmeri)$, with $\theta_L(w)$ defined analogously. 
Recall that we apply $\theta_R$ (and $\theta_L$) to a $k$-mer string $\kmeri$ as $\theta_R(\kmeri) = \theta_R(\kmerstr{i}) = \kmeri_0 \kmeri_1 \kmeri_2 \cdots \kmeri_{k-2}$.

After applying these operations to all $k$-mers we are left with $k-1$-mers, some of which are identical.
We can re-index the $k-1$-mers as $\kmerj$ for $j \in \{1, \cdots, |\cA|^{k-1}\}$ and retrieve new counts $\beta_j = \sum_{i \mid \theta_R(\kmeri) = \kmerj} \alpha_i$ to get that $\theta_R(w) = \sum_{j=1}^{|\cA|^{k-1}} \beta_j \kmerj$.
Equivalently, we have that $\theta_L(w) = \sum_{j=1}^{|\cA|^{k-1}} \gamma_j \kmerj$ for $\gamma_j = \sum_{i \mid \theta_L(\kmeri) = \kmerj} \alpha_i$.
Then $(\theta_L - \theta_R)w = \theta_L(w) - \theta_R(w) = \sum_{j=1}^{|\cA|^{k-1}} (\gamma_j - \beta_j) \kmerj$.

In $G_k$, all $k-1$-mers are vertices and $k$-mers are directed edges (weighted by the relevant $k$-mer occurrences). 
Then for a $k-1$-mer, $\kmerj$, $\beta_j$ sums over all outgoing edges for the vertex corresponding to $\kmerj$ and and $\gamma_j$ sums over all incoming edges for the same vertex. 
Since $G_k$ is Eulerian we have that in-degree $=$ out-degree for each node, thus $\beta_j = \gamma_j \ \forall \ j$. 
Therefore $(\theta_L - \theta_R)w = 0$ and so $w \in \text{ker}(\theta_L - \theta_R)$.
\end{proof}

It is known that the kernel of the incidence matrix of a graph, both undirected and directed, coincides with the cycle space of the graph. \citep{bollobas_modern}
As we have just shown, any $k$-mer count tensor is an element of $W_\cA^k$ = ker$(\inc)$.
In turn, ker$(\inc)$, being the cycle space of the \dB graph, is spanned by the set of basis vectors we have previously derived, which comprise a cycle basis for the directed \dB graph.
Further note that the dimension of the cycle space of \dB graphs, $|\cA|^k - |\cA|^{k-1} + 1$, is smaller than the total number of $k$-mers, $|\cA|^k$.

In Section \ref{sec:main_results}, we showed a construction for the basis elements of the cut space and cycle space of $G_k$.
Our construction additionally leads to a combinatorial proof of the dimension of the cycle space, based on counting the elements of $X_\cA^k$ which we use to build the basis vectors of the cycle space.
We present this argument in \ref{apdx:combinatorial_proof}.
 \section{Algebraic structure}

In the previous sections, we considered \dB graphs of order $k$ for some fixed value $k$.
In this section, we are interested in the collection of \dB graphs across all orders $k$, \textit{simultaneously}.
By definition, the graph theory operators we defined, e.g. $\theta_L, \theta_R$, act independently of $k$.
Therefore, the adjacency, incidence, and Laplacian operators act across the set of all orders of \dB graphs simultaneously.
By taking this perspective, we show that the cycle space has a richer algebraic structure than just that of a vector space.

\subsection{Background on tensor algebras and Hopf algebras}

The results of this section will be cast in the language of graded tensor algebras and graded Hopf algebras. 
We therefore begin with a brief review of some relevant concepts, mostly with an eye towards establishing notation.
A reader who is unfamiliar with tensor algebras and Hopf algebras will likely find this too quick, and we invite them to review \citet{dascalescu2000hopf} for a detailed exposition.

Let $Z$ be a vector space over a field $K$.
For a non-negative integer $k$, recall that we define the $k^{th}$-order tensor space via the tensor product of $Z$ with itself $k$ times:
\begin{align*}
    Z^k = Z \underbrace{\otimes \cdots \otimes}_{\textrm{$k-1$ times}} Z.
\end{align*}
We then define the \textit{tensor algebra} over $Z$ to be
\begin{align*}
    T(Z) = \oplus_{k=0}^\infty Z^k
\end{align*}
where $Z^0 = K$.
We note that for any tensor algebra we can define the filtration $F_k = \oplus_{i=0}^k Z^i$ such that that $F_0 \subseteq F_1 \subseteq \cdots$ and $F_i \odot F_j \subseteq F_{i+j}$ for all $i,j$.
As such, any tensor algebra is also a graded, filtered algebra.

Given an algebra on $Z$, a coalgebra can be defined if there exists $K$-linear maps $\Delta: Z \to Z \times Z$ and $\varepsilon: Z \to K$, such that
\begin{align*}
    (\text{id}_Z \otimes \Delta) \circ \Delta &= (\Delta \otimes \text{id}_Z) \circ \Delta \\
    (\text{id}_Z \otimes \varepsilon) \circ \Delta &= (\varepsilon \otimes \text{id}_Z) \circ \Delta
\end{align*}
where $\text{id}_Z$ is the identity on $Z$.
$\Delta$ is referred to as the coproduct and $\varepsilon$ is the counit.

$T(Z)$ carries a pair of well-known products and coproducts, which we define and provide examples of here.
Let $z = z_1 z_2 \cdots z_k \in Z^k, z' = z'_1 z'_2 \cdots z'_\ell \in Z^\ell$.

\begin{definition}[Concatenation product]
    We use the symbol $\concat$ to represent concatenation of two words, given by
    \begin{align*}
        z z' = z \concat z' &= z_1 z_2 \cdots z_k z'_1 z'_2 \cdots z'_\ell.
    \end{align*}
\end{definition}

\begin{definition}[Shuffle product]
\label{def:shuffle_product}
The shuffle product over two words is given by the sum over all the ways in which the letters of two words can be interleaved while preserving the letter-ordering of the individual words.

We represent this operation with the symbol $\shuffle$ and it has the following, equivalent, recursive definitions given by
\begin{align*}
    z \shuffle z' &= z_1 (z_{2\cdots k} \shuffle z') + z'_1 (z \shuffle z'_{2\cdots \ell}) \\
    &= (z_{1\cdots k-1} \shuffle z')z_{k} + (z \shuffle z'_{1\cdots \ell-1})z'_{\ell}
\end{align*}
\end{definition}
where $z_{1\cdots k-1}$ is the contiguous subsequence of $z$ obtained from the $1^{st}$ to $(k-1)^{st}$ letters of $z$, with similar definitions holding for the other indexed words given above.

\begin{example}[Shuffle product of two words]
Given the strings $ab$ and $cd$ for some letters $a, b, c, d \in Z$,
\begin{align*}
    ab \shuffle cd = abcd + acbd + cabd + acdb + cadb + cdab.
\end{align*}
\end{example}

We now provide definitions and examples for the counit and coproducts with respect to the concatenation and shuffle products.

We will refer to $\emptyword$ as the empty word, which is the unit of $T(Z)$ and acts on the concatenation and shuffle operators as $\emptyword \concat z = z \concat \emptyword = z$ and  $\emptyword \shuffle z = z \shuffle \emptyword = z$, respectively, for all $z \in T(Z)$.
The counit $\varepsilon$ is defined as $\varepsilon(\emptyword) = 1$ and $\varepsilon(z) = 0$ for any other $z \neq \emptyword$ (here we simply take the underlying field $K = \mathbb{C}$).

We use a more general notion of an indexed word in defining the coproducts.
For a word $z = z_1 z_2 \cdots z_k$ define an index set $I \subseteq \{1, \cdots, k\}$ such that $I = \{i_1, i_2, \cdots \}$ preserves order, meaning $i_1 < i_2 < \cdots$. Then we define $z_I = \concat_{i \in I} z_i$. 
That is, $z_I$ is the word made from the letters of $z$ corresponding to indices of $I$, in an order-preserving manner.
Note that $z_\emptyset = \emptyword$.

\begin{example}[Indexed word]
    Let $z = abcd$ and $I = {1, 3}$. Then $z_I = ac$.
\end{example}

The coproduct of the concatenation product is given by the de-shuffle operation, and the coproduct of the shuffle product is given by the de-concatenation operation.

\begin{definition}[De-concatenation coproduct]
    Given a word $z = z_1 z_2 \cdots z_k$, the de-concatenation operator is given by
    \begin{align*}
        \Delta_{\concat} = \sum_{i = 0}^{k} z_1 \cdots z_{i} \otimes z_{i+1} \cdots z_{k}
    \end{align*}
    where the $i=0$ term corresponds to $\emptyword \otimes z$ and $i=k$ corresponds to $z \otimes \emptyword$.
\end{definition}

\begin{definition}[De-shuffle coproduct]
    Given a word $z = z_1 z_2 \cdots z_k$, the de-shuffle operator is given by
    \begin{align*}
        \Delta_{\shuffle} = \sum_{I \subseteq \{1, \cdots, k\}} z_I \otimes z_{\{1, \cdots, k\} \setminus I}.
    \end{align*}
\end{definition}

\begin{example}[De-concatenation example]
    Let $z = abcd$. Then 
    $$\Delta_{\concat}(z) = \emptyword \otimes abcd + a \otimes bcd + ab \otimes cd + abc \otimes d + abcd \otimes \emptyword.$$
\end{example}

\begin{example}[De-shuffle example]
    Let $z = abc$. Then
    $$\Delta_{\shuffle}(z) =  \emptyword \otimes abc + a \otimes bc + b \otimes ac + c \otimes ab + ab \otimes c + ac \otimes b + bc \otimes a + abc \otimes \emptyword.$$
\end{example}

Recall that a bialgebra is a vector space that is equipped with a product and coproduct, such that the coproduct is a homomorphism of the algebra and the product is a homomorphism of the coalgebra.
A bialgebra becomes a \textit{Hopf algebra} when it also has an \textit{antipode}. The antipode is given by the inverse of the identity mapping (see Chapter 4 of \citet{dascalescu2000hopf} for a more thorough description of bialgebras and Hopf algebras). 

For $Z$ a $K$-bialgebra with product $\nabla$, unit $\eta$, coproduct $\Delta$, and counit $\varepsilon$, an antipode is a linear map $\gamma: Z \to Z$ such that
\begin{align*}
    \nabla \circ (\text{id}_Z \otimes \gamma) \circ \Delta = \nabla \circ (\gamma \otimes \text{id}_Z) \circ \Delta = \eta \circ \varepsilon.
\end{align*}

For the de-shuffle and de-concatenation coproducts on $T(Z)$ we have an antipode.
\begin{definition}[Antipode for de-shuffle and de-concatenation coproducts]
    \label{def:antipode}
    Let $z = z_1 z_2 \cdots z_k \in T(Z)$. Define $\gamma: T(Z) \to T(Z)$,
    \begin{align*}
        \gamma(z) = (-1)^k z_k z_{k-1} \cdots z_1.
    \end{align*}
\end{definition}

The following theorem is well know and can be found in, for example, \citet{Crossley2006SOMEHA} and \citet{yeats_combinatorial}.
\begin{theorem}
Let $T(Z)$ be a filtered, graded tensor algebra.
\begin{enumerate}
    \item With product $\shuffle$, coproduct $\Delta_{\concat}$, and antipode $\gamma$, $T(Z)$ is a Hopf algebra.
    \item With product $\concat$, coproduct $\Delta_{\shuffle}$, and antipode $\gamma$, $T(Z)$ is a Hopf algebra.
    \item These two Hopf algebras are dual to each other.
\end{enumerate}
\end{theorem}

When we take $Z$ to be the vector space over a discrete alphabet $\cA$ and carry forward the definitions given above, we obtain a \textit{combinatorial Hopf algebra}.
We will now proceed to consider the tensor algebras over the various vector spaces of $\dB$ graphs that we have defined in our manuscript.

\subsection{Tensor algebras associated with De Bruijn graphs}

Throughout this manuscript we have considered the following three vector spaces associated with \dB graphs or order $k$ over an alphabet $\cA$.
Recall that we take $V_\cA$ to be a vector space with an orthogonal basis indexed by the letters of $\cA$.
Then we define
\begin{align*}
    V_\cA^k &= V_\cA \underbrace{\otimes \cdots \otimes}_{\textrm{$k-1$ times}} V_\cA \\
    W_\cA^k &= \text{ker}(\theta_L - \theta_R) \subseteq V_\cA^k \\
    X_\cA^k &= \text{ker}(\theta_L) \cap \text{ker}(\theta_R) \subseteq V_\cA^k.
\end{align*}
We now consider the following filtered, graded tensor algebras given by the infinite direct sums across all values of $k$,
\begin{align*}
    T(V) &= \oplus_{k=0}^\infty V_\cA^k \\
    T(W) &= \oplus_{k=0}^\infty W_\cA^k \\
    T(X) &= \oplus_{k=0}^\infty X_\cA^k.
\end{align*}

Recall that for a $K$-algebra $Z$ with a product $\nabla$, a \textit{derivation}, $D$, is a $K$-linear map that acts on $z, z' \in Z$ as
\begin{align*}
    D\nabla(z, z') = \nabla(Dz, z') + \nabla(z, Dz')
\end{align*}
which is known as Leibniz's law.

\begin{lemma}[$\theta_L, \theta_R$ are derivations of the shuffle product]
    \label{lem:theta_derivations}
    For words $z, z' \in T(Z)$, we have that
    \begin{align*}
        \theta_L(z \shuffle z') &= \theta_L(z) \shuffle z' + z \shuffle \theta_L(z') \\
        \theta_R(z \shuffle z') &= \theta_R(z) \shuffle z' + z \shuffle \theta_R(z').
    \end{align*}
\end{lemma}
\begin{proof}
    Let $z \in Z^k, z' \in Z^\ell$.
    We start with the first recursive definition of the shuffle product, given by $z \shuffle z' = z_1(z_{2\cdots k} \shuffle z') + z'_1(z \shuffle z'_{2\cdots \ell})$.
    Then it is clear that $\theta_L(z \shuffle z') = z_{2\cdots k} \shuffle z' + z \shuffle z'_{2\cdots \ell}$..
    Further, we can see that $\theta_L(z) \shuffle z' = z_{2\cdots k} \shuffle z'$ and $z \shuffle \theta_L(z') = z \shuffle z'_{2 \cdots \ell}$. A similar argument holds for $\theta_R$ using the second recursive definition of the shuffle product.
\end{proof}

\begin{corollary}[$\theta_L - \theta_R$ is a derivation of the shuffle product]
    As linear maps, $\theta_L, \theta_R$ are derivations of the shuffle product. Therefore, so is $\theta_L - \theta_R$.
\end{corollary}

From Lemma \ref{lem:theta_derivations}, we immediately get the following result.
\begin{proposition}
$T(V), T(W), T(X)$ are graded algebras and are closed under the shuffle product.
\end{proposition}
\begin{proof}
    This is clear on $T(V)$ by the base definitions we have given.
    Take $x, x' \in T(X)$.
    Recall the definition that $x, x' \in \text{ker}(\theta_L) \cap \text{ker}(\theta_R)$.
    Then $\theta_L x = \theta_R x = \theta_L x' = \theta_R x' = \mathbf{0}$, where $\mathbf{0}$ is the zero-tensor.
    Using Lemma \ref{lem:theta_derivations}, we have that 
    \begin{align*}
        \theta_L (x \shuffle x') &= \theta_L(x) \shuffle x' + x \shuffle \theta_L(x') \\
        &= \mathbf{0} \shuffle x' + x \shuffle \mathbf{0} = \mathbf{0}.
    \end{align*}
    It should be clear that the same holds for $\theta_R$, and so $x \shuffle x' \in T(X)$.

    Now take $w, w' \in T(W)$. Then $w, w' \in \text{ker}(\theta_L - \theta_R)$.
    As $\theta_L - \theta_R$ is a derivation of the shuffle product, we similarly have that
    \begin{align*}
        (\theta_L - \theta_R)(w \shuffle w') &= (\theta_L - \theta_R)(w) \shuffle w' + w \shuffle (\theta_L - \theta_R)(w') \\
        &= \mathbf{0} \shuffle w' + w \shuffle \mathbf{0} = \mathbf{0}.
    \end{align*}
    Therefore $w \shuffle w' \in T(W)$.
\end{proof}

\begin{proposition}
    $T(V), T(X)$ are closed under the concatenation product while $T(W)$ is not.
\end{proposition}
\begin{proof}
    It is obvious that $T(V)$ is closed under $\concat$ and recalling that $\hat{x}, \hat{x}' \in T(X)$, in Fourier space, means that $\hat{x}, \hat{x}'$ neither start nor end with $\dBvec$ then it is clear that neither does $\hat{x} \concat \hat{x}'$.
    As such, $\hat{x} \concat \hat{x}' \in T(X)$.

    That $T(W)$ is not closed under concatenation can be seen by a simple counterexample.
    Take $w, w' \in T(W)$ defined as $w = aaa, w' = bbbb$ for some arbitrary, non-zero letters $a, b \in \cA$.
    Then, returning to $\theta_L, \theta_R$ as they operate in native space,
    \begin{align*}
        \theta_L(w) = aa = \theta_R(w) \\
        \theta_L(w') = bbb = \theta_R(w')
    \end{align*}
    so clearly $w, w' \in \text{ker}(\theta_L - \theta_R)$.
    However,
    \begin{align*}
        \theta_L(w \mathbin\Vert w') = aabbbb \\
        \theta_R(w \mathbin\Vert w') = aaabbb
    \end{align*}
    so $w \mathbin\Vert w' \not\in \text{ker}(\theta_L - \theta_R)$.
    As such, we see that $T(W)$ is not closed under concatenation.
\end{proof}

While $T(X)$ carries the products $\shuffle$ and $\concat$, and $T(W)$ carries the product $\shuffle$, we will see that their natural coproducts are not well defined, in the sense that they do not map elements of $X$ and $W$ back to $X$ and $W$, respectively, but rather to $V$.
We will resolve this on $T(X)$, giving us that $T(X)$ is a combinatorial Hopf algebra under both product-coproduct pairs presented above, but we will show that $T(W)$ is still not closed under the deconcatenation coproduct.
However, when considering all $k$ simultaneously, we will see that $T(X)$ is the more important space to consider when dealing with the cycle space of \dB graphs, rather than $T(W)$.
Further, that $T(X)$ is a combinatorial Hopf algebra under both (product, coproduct) pairs given by $(\shuffle, \Delta_{\concat})$ and $(\concat, \Delta_{\shuffle})$ makes it a particularly intriguing space to consider.

We start by motivating why $T(X)$ is the more important space to consider than $T(W)$.
In Section \ref{sec:space_of_kmer_counts} we considered the operator $\psi_k: \Omega \to V_\cA^k$ which mapped a circular string $\omega \in \Omega$ to a tensor representing the count of all $k$-mers in $\omega$. 
We then showed that these tensors are actually spanned by the cycle space of the $k^{th}$-order \dB graph.
That is, $Im(\psi_k) \subseteq W_\cA^k$.
In keeping with the tensor algebraic perspective, we now consider a map $\psi: \Omega \to T(W)$ which maps a circular string to a count tensor representing the count of all $k$-mers in $\omega$, for all $k$.

Before proceeding, we require the following useful lemmas.

\begin{lemma}
    \label{lem:theta_xi}
    For $\hat{x} \in X_\cA^k$ we have that $\sqrt{\frac{j+1}{j}} \theta_L \Xi^*_{j, 0}(\hat{x}) = \sqrt{\frac{j+1}{j}} \theta_R \Xi^*_{j, 0}(\hat{x}) = \Xi^*_{j-1, 0}(\hat{x})$.
\end{lemma}
\begin{proof}
    \begin{align*}
        \sqrt{\frac{j+1}{j}} \theta_L \Xi^*_{j, 0} \hat{x} &= \sqrt{\frac{j+1}{j}} \theta_L (\frac{1}{\sqrt{j+1}} \sum_{i=0}^j (\theta_L^*)^{j-i}(\theta_R^*)^i)\hat{x} \\
        &= \frac{1}{\sqrt{j}}\sum_{i=0}^{j-1} (\theta_L^*)^{j-i-1}(\theta_R^*)^i\hat{x} + \frac{1}{\sqrt{j}} \theta_L (\theta_R^*)^j\hat{x} \\
        &= \Xi^*_{j-1, 0}\hat{x} + \frac{1}{\sqrt{j}} (\theta_R^*)^j \theta_L\hat{x} \\
        &= \Xi^*_{j-1, 0}\hat{x}
    \end{align*}
    since $\theta_L \hat{x} = 0$. The proof for $\theta_R \Xi^*_{j,0}$ follows analogously.
\end{proof}

\begin{corollary}
    \label{cor:xi_star_recursion}
    As an immediate corollary to the prior lemma, we obtain the following recursion, for any $\hat{x} \in X_\cA^k$ and $i \leq j$, for any $j$,
    \begin{align*}
        (\sqrt{j+1}) \theta_R^i \Xi^*_{j, 0} \hat{x} = (\sqrt{j - i + 1}) \Xi^*_{j-i, 0}(\hat{x})
    \end{align*}
    with an equivalent expression for $\theta_L$.
\end{corollary}

\begin{lemma}
    \label{lem:theta_ij_equiv}
    Let $w \in W_\cA^k$ and $i, j$ such that $i + j \leq k$. 
    Then $\theta_L^i \theta_R^j w \in W_\cA^{k-i-j}$.
    Further, for any $i, i', j, j'$ such that $i + j = i' + j'$, $\theta_L^i \theta_R^j w = \theta_L^{i'} \theta_R^{j'} w$.
\end{lemma}
\begin{proof}
    Since $w \in W_\cA^k$ we have that $\theta_L w = \theta_R w$.
    Additionally, it should be clear that 
    \begin{align*}
        \theta_L (\theta_L - \theta_R) w = (\theta_L - \theta_R) \theta_L w = (\theta_L - \theta_R) \theta_R w = \theta_R (\theta_L - \theta_R) w = 0
    \end{align*}
    so $\theta_L w, \theta_R w \in W_\cA^{k-1}$. 
    Repeated applications of $\theta_R, \theta_L$ give us that $\theta_L^i \theta_R^j w \in W_{\cA}^{k-i-j}$.
    It should be clear that this holds for any $i', j'$ if $i + j = i' + j'$.
\end{proof}

We are now prepared to state a proposition, demonstrating that while we care about elements in the cycle space, $T(W)$, we need only be concerned with the space $T(X)$.

\begin{proposition}
    Let $w_{k+\ell} \in W_\cA^{k+\ell}$, $w_k \in W_\cA^k$, and $\hat{x}_k \in X_\cA^k$ for any $k, \ell$. 
    Then,
    \begin{align*}
        \langle w_{k+\ell}, \Xi^*_{\ell, 0}(\hat{x}_k) \rangle = \frac{\sqrt{\ell+1}}{\sqrt{|\cA|^\ell}} \langle w_k, \hat{x}_k \rangle.
    \end{align*}
\end{proposition}

\begin{proof}
    We first note that by Corollary \ref{cor:xi_star_recursion},
    \begin{align*}
        \hat{x}_k &= \sqrt{\ell - \ell + 1}\Xi^*_{\ell - \ell, 0}(\hat{x}_k)\\
        &= (\sqrt{\ell + 1})\theta_R^{\ell} \Xi^*_{\ell, 0} \hat{x}_k
    \end{align*}
    Then noting, from Lemma \ref{lem:theta_ij_equiv}, that $\theta_R^\ell = \theta_L^i \theta_R^j$ if $i + j = \ell$,
    \begin{align*}
        (\ell + 1) \hat{x}_k &= (\sqrt{\ell + 1}) \sum_{i=0}^\ell \theta_L^i \theta_R^{\ell-i} \Xi^*_{\ell, 0} \hat{x}_k \\
        \Rightarrow \hat{x}_k &= \frac{1}{\sqrt{\ell + 1}} \sum_{i=0}^\ell \theta_L^i \theta_R^{\ell-i} \Xi^*_{\ell, 0} \hat{x}_k.
    \end{align*}

    Then for any $w_k$, we have that
    \begin{align*}
        \langle w_k, \hat{x}_k \rangle &= \langle w_k, \frac{1}{\sqrt{\ell + 1}} \sum_{i=0}^\ell \theta_L^i \theta_R^{\ell-i} \Xi^*_{\ell, 0} \hat{x}_k \rangle \\
        &= \frac{1}{\sqrt{\ell + 1}} \langle w_k, \sum_{i=0}^\ell \theta_L^i \theta_R^{\ell-i} \Xi^*_{\ell, 0} \hat{x}_k \rangle \\
        &= \frac{1}{\sqrt{\ell + 1}} \langle \sum_{i=0}^{\ell} (\thetaLt)^i (\thetaRt)^{\ell-i} w_k, \Xi^*_{\ell, 0} \hat{x}_k \rangle \\
        &= \frac{1}{\sqrt{\ell + 1}} \langle \sqrt{|\cA|^\ell}w_{k+\ell}, \Xi^*_{\ell, 0} \hat{x}_k \rangle \\
        &= \frac{\sqrt{|\cA|^\ell}}{\sqrt{\ell + 1}} \langle w_{k + \ell}, \Xi^*_{\ell, 0} \hat{x}_k \rangle.
    \end{align*} 
\end{proof}

We now show why the coproducts of $\deshuffle, \deconcat$ do not provide group homomorphisms on the algebra $T(X)$ and $T(W)$.
We formally require that these group homomorphisms map $T(X) \to T(X) \otimes T(X)$ and $T(W) \to T(W) \otimes T(W)$, respectively.
In both tensor algebras we will show that the coproducts don't satisfy this as currently defined.
To do so, we will give examples with words in Fourier space, from an alphabet $\cA = \{a, b\}$. 
We arbitrarily assign the letter $\hat{a} = \dBvec$.

Consider the word $x \in T(X)$ given by $x = \hat{b} \dBvec \hat{b}$.
The coproducts operate on $x$ as,
\begin{align*}
    \deconcat(\hat{b} \dBvec \hat{b}) &= \emptyword \otimes \hat{b} \dBvec \hat{b} + \hat{b} \otimes \dBvec \hat{b} + \hat{b} \dBvec \otimes \hat{b} + \hat{b} \dBvec \hat{b} \otimes \emptyword \\ 
    \deshuffle(\hat{b} \dBvec \hat{b}) &= \emptyword \otimes \hat{b} \dBvec \hat{b} + \hat{b} \otimes \dBvec \hat{b} + \dBvec \otimes \hat{b}\hat{b} + \hat{b} \otimes \hat{b} \dBvec + \hat{b}\dBvec \otimes \hat{b} + \hat{b}\hat{b} \otimes \dBvec + \dBvec\hat{b} \otimes \hat{b} + \hat{b} \dBvec \hat{b} \otimes \emptyword.
\end{align*}
It is clear in both cases that the coproduct splits apart an element of $T(X)$ into elements $\in T(V)$ but $\not\in T(X)$, for example $\dBvec\hat{b}, \hat{b}\dBvec \not\in T(X)$. 

Now, take the word $w \in T(W)$ given by $w = \Xi^*_{1, 0}(\hat{b}\hat{b}) = \frac{1}{2}(\dBvec\hat{b}\hat{b} + \hat{b}\hat{b}\dBvec)$.
The coproducts operate on $w$ as,
\begin{align*}
    \deconcat(\dBvec\hat{b}\hat{b} + \hat{b}\hat{b}\dBvec) &= \deconcat(\dBvec\hat{b}\hat{b}) + \deconcat(\hat{b}\hat{b}\dBvec) \\
    &= \emptyword \otimes \dBvec \hat{b} \hat{b} + \dBvec \otimes \hat{b} \hat{b} + \dBvec\hat{b} \otimes \hat{b} + \dBvec\hat{b}\hat{b} \otimes \emptyword + \emptyword \otimes \hat{b}\hat{b}\dBvec + \hat{b} \otimes \hat{b} \dBvec + \hat{b}\hat{b} \otimes \dBvec + \hat{b}\hat{b}\dBvec \otimes \emptyword \\
    \deshuffle(\dBvec\hat{b}\hat{b} + \hat{b}\hat{b}\dBvec) &= \deshuffle(\dBvec\hat{b}\hat{b}) + \deshuffle(\hat{b}\hat{b}\dBvec) \\
    &= \emptyword \otimes \dBvec \hat{b}\hat{b} + 2(\dBvec \otimes \hat{b}\hat{b}) + 2(\hat{b} \otimes \dBvec \hat{b}) + 2(\dBvec\hat{b} \otimes \hat{b}) + 2(\hat{b}\hat{b} \otimes \dBvec) + \dBvec \hat{b}\hat{b} \otimes \emptyword + \\
    &\quad \emptyword \otimes \hat{b}\hat{b}\dBvec + 2(\hat{b} \otimes \hat{b}\dBvec) + 2(\hat{b}\dBvec \otimes \hat{b}) + \hat{b}\hat{b}\dBvec \otimes \emptyword.
\end{align*}
Similarly, we see that both coproducts split apart an element of $T(W)$ into elements $\in T(V)$ but $\not\in T(W)$, for example $\dBvec\hat{b}, \dBvec\hat{b}\hat{b}, \hat{b}\dBvec, \hat{b}\hat{b}\dBvec \not\in T(W)$.

We now show that we can re-index the alphabet in such a way to preserve the spaces $T(X), T(W)$ and their closure under the products of $\shuffle, \concat$, while ensuring that $\deconcat, \deshuffle$ are valid co-products on $T(X)$.
Under this re-indexing, we find that these coproducts still do not preserve the group elements of $T(W)$, and it is an open question to determine whether there are any natural coproducts on $T(W)$.
However, we have shown that $T(X)$ conveys all of the relevant information of $T(W)$ and we will now see that $T(X)$ does, indeed, form dual combinatorial Hopf algebras under the $(\shuffle, \deconcat)$ and $(\concat, \deshuffle)$ (product, coproduct) pairs.

We start by introducing the notion of a \textit{primitive word}.

\begin{definition}[Primitive words on $T(X)$]
    An primitive word on $T(X)$, in Fourier space, is any word $\in T(X)$ that cannot be represented as the concatenation of any two words in $T(X)$.
    Words of length $1$ that are $\neq \dBvec$ are also considered primitive.

\end{definition}

What this means is that an primitive word can only contain contiguous, repeat occurrences of the letter $\dBvec$. If any sequence of letters, not necessarily the same, that are $\neq \dBvec$ occur in a contiguous sequence, then this word can further be broken down into primitive words.

We give an example of primitive words, and of how this re-indexing works on $X$.
\begin{example}
    Take $\cA = \{\dBvec, \hat{b}, \hat{c} \}$.
    Example words $\in T(X)$ include: $\hat{b}\dBvec\hat{b}, \hat{b}\hat{b} \dBvec\hat{b}, \hat{b}\dBvec\dBvec\hat{c}\dBvec\hat{b}, \hat{b}\dBvec\dBvec\hat{c}\hat{b}\dBvec\hat{c}$.\\

    From the previous definition, we see that the words $\hat{b}\dBvec\hat{b}$ and $\hat{b}\dBvec\dBvec\hat{c}\dBvec\hat{b}$ are primitive, while $\hat{b}\hat{b} \dBvec\hat{b}$ and $\hat{b}\dBvec\dBvec\hat{c}\hat{b}\dBvec\hat{c}$ are not and, in fact, can be reduced as $\hat{b} \concat \hat{b}\dBvec\hat{b}$ and $\hat{b}\dBvec\dBvec\hat{c} \concat \hat{b}\dBvec \hat{c}$, respectively.\\

    Let us now consider a new alphabet, $\Tilde{\cA}$ whose letters are given by the set of all primitive words.
    For the sake of our example, we arbitrarily denote letters $\Tilde{a} = \hat{b} \dBvec \hat{b}, \Tilde{b} = \hat{b}, \Tilde{c} = \hat{b}\dBvec\dBvec\hat{c}\dBvec\dBvec\hat{c}, \Tilde{d} = \hat{b}\dBvec\hat{c}, \Tilde{e} = \hat{b}\dBvec\dBvec\Tilde{c}$.
    Now, each of $\Tilde{a}, \Tilde{b}, \Tilde{c}, \Tilde{d}, \Tilde{e}$ are primitive words.
    In this new alphabet, the four words above become: $\Tilde{a}, \Tilde{b}\Tilde{a}, \Tilde{c}, \Tilde{e}\Tilde{d}$.

\end{example}

Note that re-indexing the alphabet into primitive words does not change the spaces $T(W), T(X)$, and neither does it affect their closure under the products $\shuffle$ and $\concat$.
However, on $T(X)$ we get that $\deshuffle$ and $\deconcat$ map to $T(X) \otimes T(X)$ because the elements $\dBvec$ can no longer be split up by the operations of de-shuffling and de-concatenation.
We observed in our prior example that these operations, when applied on the base alphabet $\cA$, mapped elements of $T(X)$ to elements with $\dBvec$ on the boundaries of the word. 
When we consider primitive words as a letter then we will never obtain such an undesirable split when applying the coproducts.

To see this, consider $x \in T(X)$ where $x = \hat{b}\dBvec\dBvec\hat{c}\hat{b}\dBvec\hat{c}$.
The actions of deshuffling and deconcatenation, with respect to primitive words, are given by
\begin{align*}
    \deconcat(\hat{b}\dBvec\dBvec\hat{c}\hat{b}\dBvec\hat{c}) &= \emptyword \otimes \hat{b}\dBvec\dBvec\hat{c}\hat{b}\dBvec\hat{c} + \hat{b}\dBvec\dBvec\hat{c} \otimes \hat{b}\dBvec\hat{c} + \hat{b}\dBvec\dBvec\hat{c}\hat{b}\dBvec\hat{c} \otimes \emptyword \\
    \deshuffle (\hat{b}\dBvec\dBvec\hat{c}\hat{b}\dBvec\hat{c}) &= \emptyword \otimes \hat{b}\dBvec\dBvec\hat{c}\hat{b}\dBvec\hat{c} +  \hat{b}\dBvec\dBvec\hat{c} \otimes \hat{b}\dBvec\hat{c} + \hat{b}\dBvec\hat{c} \otimes \hat{b}\dBvec\dBvec\hat{c} + \hat{b}\dBvec\dBvec\hat{c}\hat{b}\dBvec\hat{c} \otimes \emptyword.
\end{align*}
We immediately see that when primitive words are kept together, the coproducts both map $T(X) \to T(X) \otimes T(X)$.
Equipped with the antipode given in Definition \ref{def:antipode}, we observe that $T(X)$ forms dual combinatorial Hopf algebras with respect to $(\shuffle, \deconcat)$ and $(\concat, \deshuffle)$.

 \section{Discussion}

In this work we studied the vector spaces of De Bruijn graphs.
To do so, we introduced a novel set of operators, $\theta_L$ and $\theta_R$, that allow us to translate the graph adjacency, incidence, and Laplacian matrices into operators on tensor spaces.
When we view the space of all $k$-mers as a tensor space, $V_\cA^k$, the graph operators $A, \inc, \Lambda_V$, and $\Lambda_E$ are seen to simultaneously operate over all possible $k$.
Through a Fourier analysis of these operators we are able to obtain a closed-form expression for the eigenvectors of both $\Lambda_V$ and $\Lambda_E$, which further decompose into a basis for the cut space, $(W_\cA^k)^{\perp}$ and cycle space, $W_\cA^k$ of all De Bruijn graphs.

We finally show that there is a subspace of the cycle space, $X_\cA^k$ that fully characterizes the cycle space itself.
That is the subspace of Fourier words, as defined by those words, in Fourier space, that neither start nor end with the special character, $\dBvec$.
When considering the direct sum of all the vector spaces $X_\cA^k$, we obtain a graded, filtered tensor algebra, $T(X)$.
Operating over the alphabet of primitive words on $T(X)$, this space carries the structure of dual combinatorial Hopf algebras under the products of concatenation, $\concat$, and shuffle, $\shuffle$, with coproducts given by de-shuffle, $\deshuffle$, and de-concatenation, $\deconcat$, respectively.\\

\paragraph{Future work:} One important future direction is to define a discrete exterior calculus on the edge space of De Bruijn graphs and expand on concepts related to cohomology and discrete differential geomtry.
We additionally see future work connecting the operators defined here to other algebras, for example Rota-Baxter algebras, as well as understanding whether there exists any compatible coproducts on the cycle space.
Another interesting direction is to connect further to modern graph theory perspectives, such as finding the Ihara-Zeta function for De Bruijn graphs and examining what this means.

Another natural extension would be to move from discrete operators to continuous ones. 
Naturally, the discrete Fourier transform here could be replaced by it's continuous counterpart, and analogous definitions from $\theta_L, \theta_R$ can be defined.
In this way, we would see De Bruijn graphs playing a larger rule in generalized time-series and sequential data analysis.
It would also be interesting to define an analogous graph structure on higher dimensional data beyond the 1d case, such as images, etc.
To the best of our knowledge, there has not been any study of a De Bruijn graph equivalent on 2d image patches.
Finally, the work here presents a basis that can fully represent $k$-mer count tensors.
In machine learning, natural language processing, and bioinformatics, string kernels on $k$-mer counts are commonly used, for some fixed values of $k$, and have in the past been computed using suffix trees and other algorithmic techniques.
We are hopeful that this analysis will pave the way to richer string kernel designs, and perhaps even a closed-form expression for a kernel function that measures similarity between circular strings across all $k$ simultaneously.  
\section{Acknowledgements}

We would like to thank Glenn Tesler for his thorough review of the manuscript and discussion of the ideas in this paper.

\section{Author Contributions}
Following the CRediT taxonomy\footnote{As laid our here: \url{https://beta.elsevier.com/researcher/author/policies-and-guidelines/credit-author-statement}.},\\

\textbf{Anthony Philippakis:} Conceptualization, Formal analysis, Funding acquisition, Methodology, Investigation, Writing - Original Draft, Writing - Review \& Editing, Supervision, Project administration, Validation.
\textbf{Neil Mallinar:} Methodology, Formal analysis, Investigation, Writing - Original Draft, Writing - Review \& Editing, Visualization, Validation.
\textbf{Parthe Pandit:} Formal analysis, Investigation, Writing - Original Draft, Writing - Review \& Editing, Validation.
\textbf{Mikhail Belkin:} Supervision, Validation.

\appendix
\section{More on deletion operators}
\label{apdx:deletion_ops}

We first recall the definition of the deletion operators and their adjoints.

\begin{definition}[Deletion operators] 
For $k=1$, returns scalar 1, $k=0$, returns 0.
For $k\geq 2$, we have the left and right delete operators $\theta_L$ and $\theta_R$ as linear maps from $\VA^k\mapsto \VA^{k-1}$, defined as follows,
\begin{align}
    \theta_L(v) &= \frac{1}{\sqrt{|\cA|}} v_1\ldots v_{k-1}\in \VA^{k-1}\qquad\forall v \in\VA^k \\
    \theta_R(v) &= \frac{1}{\sqrt{|\cA|}} v_0\ldots v_{k-2}\in \VA^{k-1}\qquad\forall v \in\VA^k.
\end{align}

These operators have their own adjoint operators, which are linear maps from $\VA^{k-1}\mapsto\VA^{k}$ given by, 
\begin{align}
    \theta^*_L(v) &= \dBvec v \in\VA^k \qquad\forall v \in\VA^{k-1} \\
    \theta^*_R(v) &= v \dBvec \in\VA^k \qquad\forall v \in\VA^{k-1}
\end{align}
where $\dBvec$ is the de Bruijn vector defined in Definition \ref{def:dbvec}.
\end{definition}

Intuitively, we see that the transpose deletion operators, $\theta_L^*$ and $\theta_R^*$, act to insert all possible letters from $\cA$ into the leftmost or rightmost position of a $k$-mer, resulting in all possible $(k+1)$-mers that can be obtained from a given $k$-mer.

We provide simple examples to illustrate why the deBruijn vector, $\dBvec$, is of importance and how it is used to define the deletion operators.
In the case $k=2$ we want that $\theta_L(v_0 \otimes v_1) = v_1$.
Here we see that the tensor product is given by the outer product $v_0 \otimes v_1 = v_0 v_1^T$.
A linear operator $\theta_L$ must then serve as a vector $v$ such that $v^Tv_0v_1^T = v_1^T$. 
This requires $v^Tv_0 = 1 \ \forall \ v_0 \in \VA$.
We immediately see that the only such vector is given by $v = \dBvec$ when $\VA$ is given by the standard basis.
Similarly we have that $\theta_R(v_0 \otimes v_1) = v_0 \Rightarrow v_0v_1^Tv = v_0$ which requires that $v_1^T v = 1 \ \forall v_1 \ \in \VA$. Again, the vector $v = \dBvec$ is the only candidate.

From this example we can see that the generalized operators, for any $k$, are given by 
\begin{align}
    \label{eqn:thetaL_formal_dfn}
   \theta_L(v_0 \otimes v_1 \otimes \cdots \otimes v_{k-1}) &= \frac{1}{\sqrt{|\cA|}}\langle \dBvec, v_0 \rangle (v_1 \otimes \cdots \otimes v_{k-1}) \\
   \label{eqn:thetaR_formal_dfn}
   \theta_R(v_0 \otimes v_1 \otimes \cdots \otimes v_{k-1}) &= \frac{1}{\sqrt{|\cA|}}\langle \dBvec, v_{k-1} \rangle(v_0 \otimes v_1 \otimes \cdots \otimes v_{k-2})
\end{align}
and the adjoint operators are well-defined, as given $v^{(1)} \in \VA^k, v^{(2)} \in \VA^{k-1}$ the adjoints must satisfy that $\langle \theta_L(v^{(1)}), v^{(2)} \rangle = \langle v^{(1)}, \theta_L^*(v^{(2)}) \rangle$ and $\langle \theta_R(v^{(1)}), v^{(2)} \rangle = \langle v^{(1)}, \theta_R^*(v^{(2)}) \rangle$.

Immediately we notice the following identities on $\theta_L, \theta_R$:
\begin{align*}
    \theta_L \theta_L^* = \theta_R \theta_R^* = I \\
    (\theta_L^* \theta_L)^2 = \theta_L^* \theta_L = (\theta_L^* \theta_L)^* \\
    (\theta_R^* \theta_R)^2 = \theta_R^* \theta_R = (\theta_R^* \theta_R)^* 
\end{align*}
additionally showing that $\theta_L^* \theta_L, \theta_R^* \theta_R$ are orthogonal projections.
We also note that that these operators commute in the following ways: 
\begin{align*}
    \theta_L\theta_R &= \theta_R\theta_L \\
    \theta_L^*\theta_R^* &= \theta_R^*\theta_L^* \\
    \theta_L\theta_R^* &= \theta_R^*\theta_L \\
    \theta_L^*\theta_R &= \theta_R\theta_L^*.
\end{align*} \section{Combinatorial proof of \tops{$\text{dim}(\cC(G_k)) = |\cA|^k - |\cA|^{k-1} + 1$}{dim(C(Gk))}}
\label{apdx:combinatorial_proof}

We provide a combinatorial proof for the number of basis vectors that span the cycle space of $G_k$ with $|\cA| = q$.
For each $a \in \cA$ we have an equivalent character in Fourier space, $f_a \in \C^q$ (equivalently denoted as $\hat{a}$).
We will denote a special character $f_0 = \mathbf{1}_q = \dBvec$ by the all-ones vector, or the first column of the equivalent Fourier and Hadamard basis matrices $F_q, H_q$, respectively.
We arbitrarily select a letter $a \in \cA$ to be represented by $f_0$ and the remainder of the $q-1$ letters given by $f_1, \cdots f_{q-1}$, such that $f_i \perp f_j$ for $i \neq j$.

Let $X_\cA^r$ represent all ``Fourier words" of length $r$ composed by characters in Fourier space such that the word neither starts nor ends with $f_0$.
(Every element of $X_\cA^r$ is a summation over all $r$-mers from $\cA$.)

For example, 
\begin{align*}
X_\cA^0 = \{ \emptyset \} && X_\cA^1 = \{f_1, f_2, \cdots, f_{q-1}\} && X_\cA^2 = \{ f_1 f_1, \cdots, f_1 f_{q-1}, \cdots, f_{q-1}f_{q-1} \},
\end{align*}
and so on for $r \leq k$. Note that $f_if_j = f_i \otimes f_j$ as defined by the tensor product operation.
Then we obtain a basis vector of the cycle space of $G_k$ for each $x \in X_\cA^r$, for all $r$.
The total number of basis vectors is then given by $\sum_{r=0}^{k} |X_\cA^r|$ (here we consider $|X_\cA^0| = 1$).

We will now count the size of each $|X_\cA^r|$ for arbitrary $k$.
There are two base cases we count easily.
For $r=0$ there is precisely one basis vector given by padding out the empty word $\emptyset$ with $k$ copies of $f_0$.
For $r=1$ there are $q-1$ words that do not start or end with $f_0$, namely: $f_1, \cdots, f_{q-1}$.
We obtain a basis vector for each of them.

For a word of length $r$ such that $2 \leq r \leq k$, we can fill the $r-2$ inner letters of the word with any of the $q$ characters $\{f_0,\ldots, f_{q-1}\}$, whereas the first and last characters can be $f_0$ as well.
Therefore, for a fixed $r$ we would have $q^{r-2}(q-1)^2$ Fourier words that neither start nor end with $f_0$.
Summing over all possible length $r$ words with $r \in [2, k]$, we have that
\begin{align*}
    \sum_{r=2}^{k} |X_\cA^r| &= \sum_{r=2}^k q^{r-2}(q-1)^2 = (q-1)^2 \sum_{r=2}^k q^{r-2} = (q-1)^2 \sum_{r=0}^{k-2} q^r \\
    &= (q-1)^2\Bigg[\frac{1-q^{k-2+1}}{1-q}\Bigg] = (q-1)^2\Bigg[\frac{q^{k-1}-1}{q-1}\Bigg] \\
    &= (q^{k-1}-1)(q-1) = q^k - q^{k-1} - q + 1
\end{align*}

In sum, the total number of basis vectors is given by
\begin{align*}
    |X_\cA^0| + |X_\cA^1| + \sum_{r=2}^{k} |X_\cA^r| = 1 + (q-1) + q^k - q^{k-1} - q + 1 = q^k - q^{k-1} + 1.
\end{align*} \section{Examples of cycle- and cut-space bases of de Bruijn graphs}
\label{app:sample_bases}

We provide examples of cycle basis and cut basis elements for a few values of $q (= |\cA|)$ and $k$.

Recall that 
\begin{align*}
    \Xi^*_{j, h}(v) = \frac{1}{\sqrt{j+1}} \sum_{\ell=0}^{j} \cos \Bigg( \frac{(2\ell+1)h\pi}{2(j+1)} \Bigg) (\thetaLt)^{j-\ell}(\thetaRt)^{\ell}(v),
\end{align*}
$X_\cA^k = \text{ker}(\theta_L) \cap \text{ker}(\theta_R)$, and that the cycle space, $W_\cA^k$, and cut space, $(W_\cA^k)^{\perp}$, are given by
\begin{align*}
    W_\cA^k &= \Xi^*_{k, 0}(X_\cA^0) \oplus \Xi^*_{k-1, 0}(X_\cA^{1}) \oplus \Xi^*_{k-2, 0}(X_\cA^2) \oplus \cdots \oplus \Xi^*_{0, 0}(X_\cA^k) \\
(W_\cA^k)^{\perp} &= [\oplus_{h=1}^{k-1} \Xi^*_{k-1, h}(X_\cA^1)] \oplus [\oplus_{h=1}^{k-2}\Xi^*_{k-2, h}(X_\cA^2)] \oplus \cdots \oplus [\oplus_{h=1}^2 \Xi^*_{2, h}(X_\cA^{k-2})] \oplus [\Xi^*_{1, 1}(X_\cA^{k-1})].
\end{align*}

\subsection*{Basis for $q = 2, k = 2$}

$\cA = \{ a, b \}, \hat{\cA} = \{\dBvec, \hat{b}\} = \{ \frac{1}{\sqrt{2}} (a + b), \frac{1}{\sqrt{2}} (a - b) \}$.\\

\paragraph{Elements of \texorpdfstring{$X_\cA^k$}{X_A^k}}

\begin{align*}
    X_\cA^0 &= \{ \emptyword \} \\
    X_\cA^1 &= \{ \hat{b} \} \\
    X_\cA^2 &= \{ \hat{b}\hat{b} \}
\end{align*}

\paragraph{Cycle basis}

\begin{align*}
    \Xi^*_{2, 0}(\emptyword) &= \frac{3}{\sqrt{3}} (\dBvec\dBvec) \\
    \quad &= \frac{3}{2\sqrt{3}}(aa + ab + ba + bb)\\
    \Xi^*_{1, 0}(\hat{b}) &= \frac{1}{\sqrt{2}}(\dBvec\hat{b} + \hat{b}\dBvec) \\
    &= \frac{1}{\sqrt{2}}(aa - bb) \\
    \Xi^*_{0, 0}(\hat{b}\hat{b}) &= \hat{b}\hat{b} \\
    &= \frac{1}{2}(aa - ab - ba + bb)
\end{align*}

\paragraph{Cut basis}

\begin{align*}
    \Xi^*_{1, 1}(\hat{b}) &= \frac{1}{\sqrt{2}}(\cos(\frac{\pi}{4})\dBvec\hat{b} + \cos(\frac{3\pi}{4})\hat{b}\dBvec)
\end{align*}

\subsection*{Basis for $q = 2, k = 3$}

$\cA = \{ a, b \}, \hat{\cA} = \{\dBvec, \hat{b}\} = \{ \frac{1}{\sqrt{2}} (a + b), \frac{1}{\sqrt{2}}(a - b) \}$.\\

\paragraph{Elements of $X_\cA^k$}

\begin{align*}
    X_\cA^0 &= \{ \emptyword \} \\
    X_\cA^1 &= \{ \hat{b} \} \\
    X_\cA^2 &= \{ \hat{b}\hat{b} \} \\
    X_\cA^3 &= \{ \hat{b}\hat{b}\hat{b}, \hat{b}\dBvec\hat{b} \}
\end{align*}

\paragraph{Cycle basis}

\begin{align*}
    \Xi^*_{3, 0}(\emptyword) &= 2(\dBvec\dBvec\dBvec) \\
    \Xi^*_{2, 0}(\hat{b}) &= \frac{1}{\sqrt{3}}(\dBvec\dBvec\hat{b} + \dBvec\hat{b}\dBvec + \dBvec\dBvec\hat{b}) \\
    \Xi^*_{1, 0}(\hat{b}\hat{b}) &= \frac{1}{\sqrt{2}}(\dBvec\hat{b}\hat{b} + \hat{b}\hat{b}\dBvec) \\
    \Xi^*_{0, 0}(\hat{b}\dBvec\hat{b}) &= \hat{b}\dBvec\hat{b} \\
    \Xi^*_{0, 0}(\hat{b}\hat{b}\hat{b}) &= \hat{b}\hat{b}\hat{b}
\end{align*}

\paragraph{Cut basis}

\begin{align*}
    \Xi^*_{2, 1}(\hat{b}) &= \frac{1}{\sqrt{3}}(\cos(\frac{\pi}{6})\dBvec\dBvec\hat{b} + \cos(\frac{3\pi}{6})\dBvec\hat{b}\dBvec + \cos(\frac{5\pi}{6})\dBvec\dBvec\hat{b}) \\
    \Xi^*_{2, 2}(\hat{b}) &= \frac{1}{\sqrt{3}}(\cos(\frac{2\pi}{6})\dBvec\dBvec\hat{b} + \cos(\frac{6\pi}{6})\dBvec\hat{b}\dBvec + \cos(\frac{10\pi}{6})\dBvec\dBvec\hat{b}) \\
    \Xi^*_{1, 1}(\hat{b}\hat{b}) &= \frac{1}{\sqrt{2}}(\cos(\frac{\pi}{4})\dBvec\hat{b}\hat{b} + \cos(\frac{3\pi}{4})\hat{b}\hat{b}\dBvec)
\end{align*}

\subsection*{Basis for $q = 4, k = 2$}

$\cA = \{ a, b, c, d \}, \hat{\cA} = \{\dBvec, \hat{b}, \hat{c}, \hat{d}\} = \{ \frac{1}{\sqrt{4}}(a + b + c + d), \frac{1}{\sqrt{4}}(a - b + c - d), \frac{1}{\sqrt{4}}(a + b - c - d), \frac{1}{\sqrt{4}}(a - b  -c + d)\}$.\\

\paragraph{Elements of $X_\cA^k$}

\begin{align*}
    X_\cA^0 &= \{ \emptyword \} \\
    X_\cA^1 &= \{ \hat{b}, \hat{c}, \hat{d} \} \\
    X_\cA^2 &= \{ \hat{b}\hat{b}, \hat{b}\hat{c}, \hat{b}\hat{d}, \hat{c}\hat{b}, \hat{c}\hat{c}, \hat{c}\hat{d}, \hat{d}\hat{b}, \hat{d}\hat{c}, \hat{d}\hat{d} \}
\end{align*}

\paragraph{Cycle basis}

\begin{align*}
    \Xi^*_{2, 0}(\emptyword) &= \frac{3}{\sqrt{3}}(\dBvec\dBvec) \\
    \Xi^*_{1, 0}(\hat{b}) &= \frac{1}{\sqrt{2}}(\dBvec\hat{b} + \hat{b}\dBvec)\\
    \Xi^*_{1, 0}(\hat{c}) &= \frac{1}{\sqrt{2}}(\dBvec\hat{c} + \hat{c}\dBvec) \\
    \Xi^*_{1, 0}(\hat{d}) &= \frac{1}{\sqrt{2}}(\dBvec\hat{d} + \hat{d}\dBvec)\\
    \Xi^*_{0, 0}(\hat{b}\hat{b}) &= \hat{b}\hat{b}\\
    \Xi^*_{0, 0}(\hat{b}\hat{c}) &= \hat{b}\hat{c}\\
    \Xi^*_{0, 0}(\hat{b}\hat{d}) &= \hat{b}\hat{d}\\
    \Xi^*_{0, 0}(\hat{c}\hat{b}) &= \hat{c}\hat{b}\\
    \Xi^*_{0, 0}(\hat{c}\hat{c}) &= \hat{c}\hat{c}\\
    \Xi^*_{0, 0}(\hat{c}\hat{d}) &= \hat{c}\hat{d}\\
    \Xi^*_{0, 0}(\hat{d}\hat{b}) &= \hat{d}\hat{b}\\
    \Xi^*_{0, 0}(\hat{d}\hat{c}) &= \hat{d}\hat{c}\\
    \Xi^*_{0, 0}(\hat{d}\hat{d}) &= \hat{d}\hat{d}\\
\end{align*}

\paragraph{Cut basis}

\begin{align*}
    \Xi^*_{1, 1}(\hat{b}) &= \frac{1}{\sqrt{2}}(\cos(\frac{\pi}{4})\dBvec\hat{b} + \cos(\frac{3\pi}{4})\hat{b}\dBvec)\\
    \Xi^*_{1, 1}(\hat{c}) &= \frac{1}{\sqrt{2}}(\cos(\frac{\pi}{4})\dBvec\hat{c} + \cos(\frac{3\pi}{4})\hat{c}\dBvec)\\
    \Xi^*_{1, 1}(\hat{d}) &= \frac{1}{\sqrt{2}}(\cos(\frac{\pi}{4})\dBvec\hat{d} + \cos(\frac{3\pi}{4})\hat{d}\dBvec)\\
\end{align*} 
\bibliographystyle{elsarticle-num-names} 
 \bibliography{main}

\end{document}